\documentclass[12pt]{amsart}
\usepackage{amsmath,amssymb,amsfonts,amsthm,amscd,graphicx,psfrag,epsfig}
\newtheorem{theorem}{Theorem}[section]

\newtheorem{theorem-definition}[theorem]{Theorem-Definition}
\newtheorem{theorem-construction}[theorem]{Theorem-Construction}
\newtheorem{lemma-definition}[theorem]{lemma--Definition}
\newtheorem{lemma-construction}[theorem]{lemma--Construction}
\newtheorem{lemma}[theorem]{Lemma}
\newtheorem{proposition}[theorem]{Proposition}
\newtheorem{corollary}[theorem]{Corollary}
\newtheorem{conjecture}[theorem]{Conjecture}
\newtheorem{definition}[theorem]{Definition}

\newcommand{\old}[1]{}

\newcommand{\Z}{{\mathbb Z}}

\newcommand{\G}{{\mathcal G}}
\newcommand{\R}{{\mathbb R}}

\newcommand{\C}{{\mathbb C}}
\newcommand{\T}{{\mathbb T}}

\newcommand{\ra}{\rightarrow}
\newcommand{\be}{\begin{equation}}
\newcommand{\ee}{\end{equation}}
\newcommand{\bt}{\begin{theorem}}

\newcommand{\et}{\end{theorem}}
\newcommand{\bd}{\begin{definition}}
\newcommand{\ed}{\end{definition}}
\newcommand{\bp}{\begin{proposition}}
\newcommand{\ep}{\end{proposition}}

\newcommand{\bl}{\begin{lemma}}
\newcommand{\el}{\end{lemma}}
\newcommand{\bc}{\begin{corollary}}
\newcommand{\ec}{\end{corollary}}
\newcommand{\bcon}{\begin{conjecture}}
\newcommand{\econ}{\end{conjecture}}

\usepackage{graphicx}
\usepackage{amsrefs}
\usepackage{geometry}
\usepackage{graphicx}
\usepackage{bm}
\usepackage{amsmath,amsthm,amssymb,graphicx,mathtools,tikz-cd,hyperref}
\usepackage[]{subcaption}
\usetikzlibrary{positioning}
\usepackage{amssymb}

\begin{document}


\title{Grove arctic curves from periodic cluster modular transformations}
\author{Terrence George} 
\address{Brown University\\Providence, RI-02906,USA}
\email{gterrence@math.brown.edu}

\maketitle
\begin{abstract}
Groves are spanning forests of a finite region of the triangular lattice that are in bijection with Laurent monomials that arise in solutions of the cube recurrence. We introduce a large class of probability measures on groves for which we can compute exact generating functions for edge probabilities. Using the machinery of asymptotics of multivariate generating functions, this lets us explicitly compute arctic curves, generalizing the arctic circle theorem of Petersen and Speyer. Our class of probability measures is sufficiently general that the limit shapes exhibit all solid and gaseous phases expected from the classification of EGMs in the resistor network model.
\end{abstract}
\bibliographystyle{amsxport}
\tableofcontents

\section{Introduction}
A function $f:\mathbb{Z} ^3 \rightarrow \mathbb{C}$ satisfies the \textit{cube recurrence}(also known as the \textit{Miwa equation} or the \textit{discrete BKP equation}(\cite{Miwa82})) if for all $(i,j,k) \in \mathbb{Z}^3$
$$
f_{i,j,k}f_{i-1,j-1,k-1}=f_{i-1,j,k}f_{i,j-1,k-1}+f_{i,j-1,k}f_{i-1,j,k-1}+f_{i,j,k-1}f_{i-1,j-1,k}.
$$
We denote by $\mathcal{F}$ the set of functions satisfying the cube recurrence.
Define the \textit{lower cone} of $(i,j,k) \in \mathbb{Z}^3$ to be $C(i,j,k):=\{(i',j',k')\in \mathbb{Z}^3_{\leq 0}:i' \leq i,j' \leq j, k' \leq k\}$. Let $\mathcal{L}$ be a subset of $\mathbb{Z}^3_{\leq 0}$ such that $\mathbb{Z}^3_{\leq 0} \setminus \mathcal{L}$ is finite and if $(i,j,k) \in \mathcal{L}$ then we have $C(i,j,k) \subseteq \mathcal{L}$. Let $\mathcal{U}:=\mathbb{Z}^3_{\leq 0} \setminus \mathcal{L}$. A \textit{set of initial conditions} is defined to be $\mathcal{I}:=\{ (i,j,k) \in \mathcal{L}:(i+1,j+1,k+1) \notin \mathcal{L} \}$. Let $\mathfrak{I}$ denote the set of all sets of initial conditions. The set of initial conditions corresponding to $\mathcal{L}=\{(i,j,k)\in \mathbb{Z}_{\leq 0}: i+j+k \leq 1-n\}$ will be denoted by  $I(n)$ and is called the standard initial conditions of order $n$.

\smallskip

If we assign formal variables $f_{i,j,k}:=x_{i,j,k}$ for ${(i,j,k)}$ in a set of initial conditions and solve for $f_{i,j,k}$ where $(i,j,k) \in \mathcal{U}$, we obtain rational functions in $x_{i,j,k}$. 
In \cite{FZ01}, Fomin and Zelevinsky showed using cluster algebra techniques that these rational functions are Laurent polynomials in $x_{i,j,k}$ with coefficients in $\mathbb{Z}$.

\smallskip

In \cite{CS04}, Carroll and Speyer studied a more general version of the cube recurrence, which they call the edge-variables version. Define  \textit{edge variables} $a_{i,j},b_{i,k},c_{i,j}$ for each $i,j,k \in \mathbb{Z}_{\leq 0}$. A function $g:\mathbb{Z}_{\leq 0} ^3 \rightarrow \mathbb{R}_{> 0}$ satisfies the \textit{edge-variables version} of the cube recurrence if
$$
g_{i,j,k}g_{i-1,j-1,k-1}=b_{i,k}c_{i,j}g_{i-1,j,k}g_{i,j-1,k-1}\\+a_{i,k}c_{i,j}g_{i,j-1,k}g_{i-1,j,k-1}+a_{j,k}b_{i,k}g_{i,j,k-1}g_{i-1,j-1,k},
$$
for $(i,j,k) \in \mathbb{Z}^3_{\leq 0}$.
Carroll and Speyer constructed combinatorial objects called \textit{groves}(See Figure \ref{grove} (a) for an example), which they showed are in bijection with the monomials in the Laurent polynomial generated by the edge-variables version of the cube recurrence. This was used to give a combinatorial proof of the Laurent property.
\begin{figure}
\centering
\subcaptionbox{A grove on $I(3)$.}
{\includegraphics[width=0.3\textwidth]{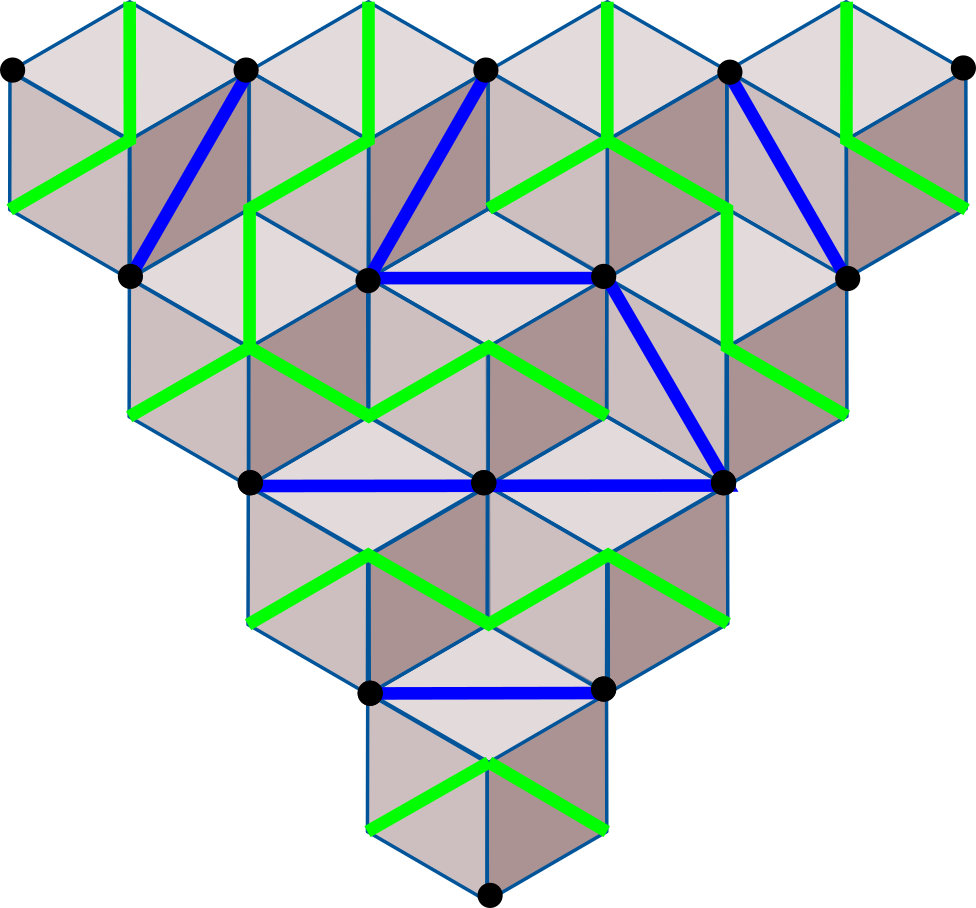}}
\subcaptionbox{The corresponding  simplified grove.}
{\includegraphics[width=0.3\textwidth]{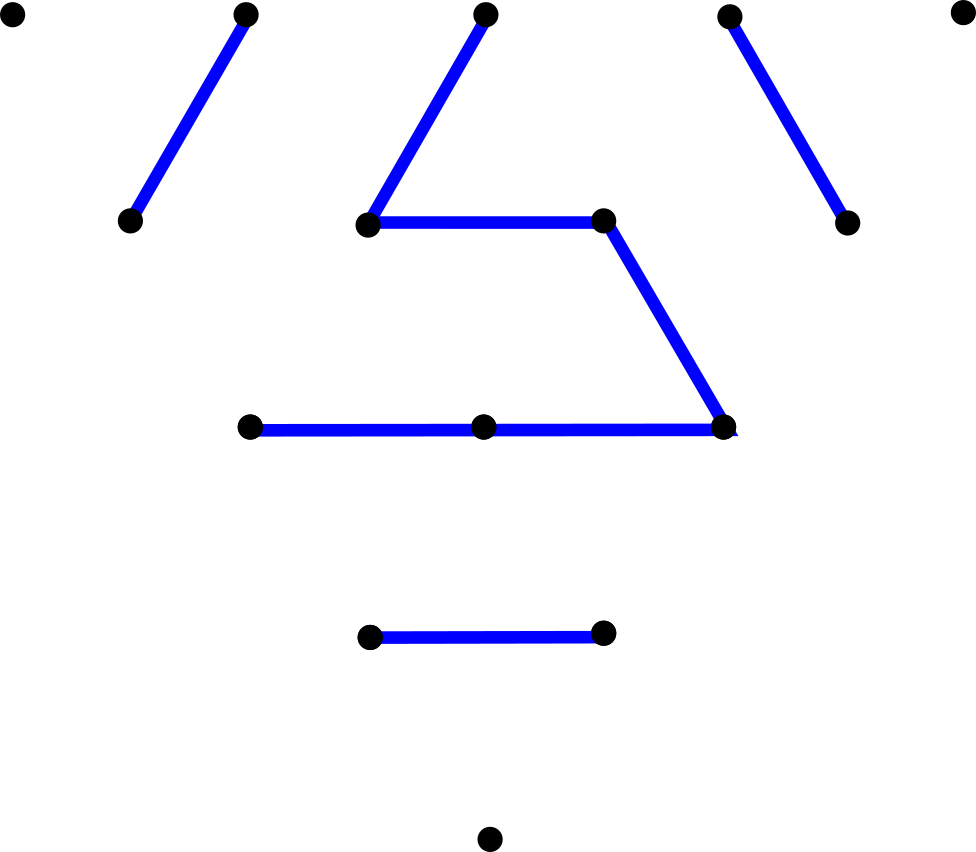}}
\caption{}\label{grove}
\end{figure}

Groves on the standard initial conditions $I(n)$ are in bijection with with spanning forests of a portion of the triangular lattice where each component of the forest connects boundary vertices in a prescribed manner (see Figure \ref{grove} (b)). Petersen and Speyer (\cite{PS05}) proved an arctic circle theorem for groves: For large $n$, a uniformly random simplified grove on $I(n)$, rescaled by a factor of $n$ so that it is now supported on the unit triangle, appears deterministic outside the inscribed circle.\\
\begin{figure}
\centering
\subcaptionbox{}
{\includegraphics[width=0.7\textwidth]{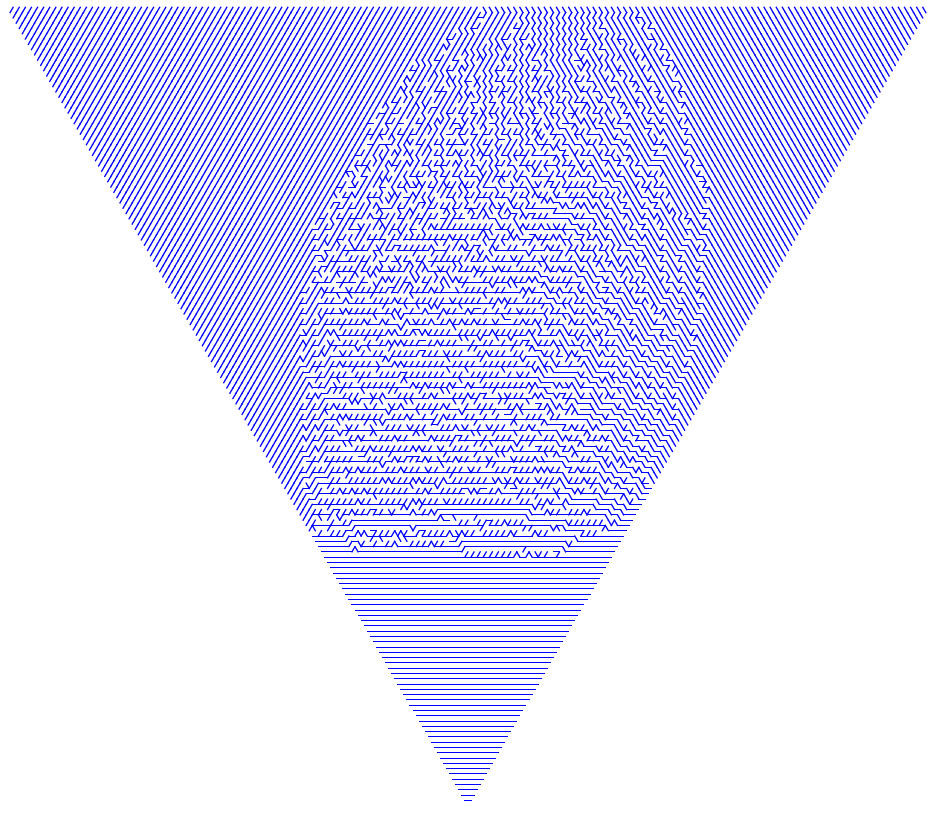}}
\subcaptionbox{The arctic curve, along with macroscopic regions labeled according to the points of the Newton polygon that correspond to the EGM describing local statistics in the region.}
{\includegraphics[width=0.7\textwidth]{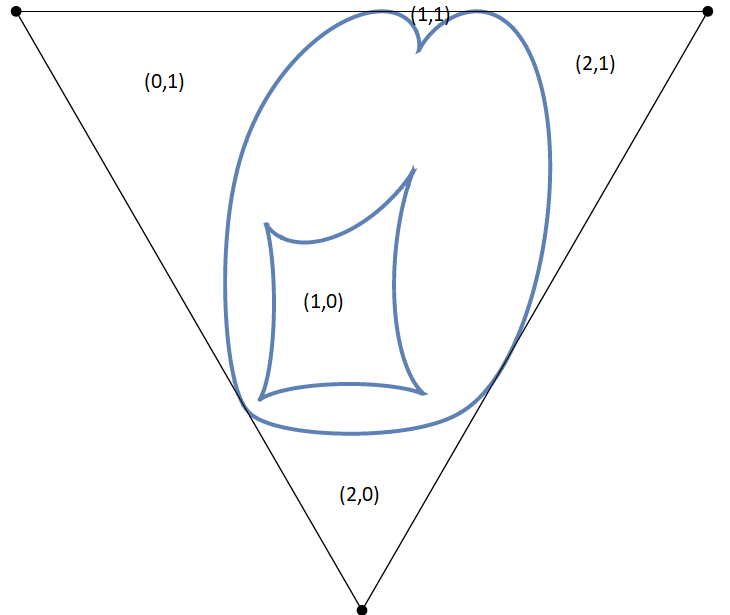}}
\caption{}\label{et}
\end{figure}
In the present paper, we extend the arctic circle theorem to a large class of probability measures on groves. There are two natural probability measures one can consider on groves:\\
\begin{itemize}
\item Given a positive real valued function $f$ satisfying the cube recurrence, we can put a probability measure on groves where each grove gets a probability proportional to the value of the monomial associated to it in the bijection of Carroll and Speyer. We denote this probability measure by $\mathbb{P}^f_{\mathcal{I}}$.
\item We can define a conductance function $C$, a positive real valued function on the edges of the triangular lattice, and consider the Boltzmann distribution, assigning to a grove $G$ the probability,
$$
\mathbb{P}^C_{\mathcal{I}}(G) \propto \prod_{\text{Edges }e \in G}C(e).
$$This is the natural measure to put on spanning forests from the point of view of statistical mechanics and generalizes the spanning tree measure.
\end{itemize}

There is a way to associate a conductance function $C^f$ to a function $f:\mathbb{Z}^3\ra \mathbb{R}$ due to Fomin and Zelevinsky(\cite{FZ01}, see also \cite{GK12}), such that the cube recurrence for $f$ becomes the resistor network $Y-\Delta$ transformation(due to Kennelly (\cite{Kennelly})) for $C^f$. We show that under this change of variables, the probability measures  $\mathbb{P}^f_{\mathcal{I}}$ and $\mathbb{P}^C_{\mathcal{I}}$ coincide (see Theorem \ref{mainthm}) and that the map $f \mapsto C^f$ is surjective. This lets us define our class of probability measures on groves in terms of conductance functions, but still allows us to exploit the algebraic structure of the $f$ functions and the cube recurrence to compute the edge probability generating functions as in \cite{PS05}.\\

Our class of probability measures come from periodic conductance functions on the triangular lattice. This however leads to an infinite system of linear equations for the edge probability generating functions. We further impose the condition that the conductance function is periodic under a cluster modular transformation (defined in Section \ref{cmtsection}) to obtain a finite linear system(Theorem \ref{pgenthm}).\\ 

We derive asymptotic edge probabilities using the machinery developed by Baryshnikov, Pemantle and Wilson(\cite{PW02},\cite{PW04},\cite{BP11} and \cite{PW13}). Our generating functions have isolated singularities with degree greater than $2$ and therefore fall outside the class of quadratic singularities considered in \cite{BP11}, but for specific examples, we see that their techniques still work. This in particular leads to explicit computations of arctic curves (see for example Figure \ref{et} (b)). \\

By analogy with the dimer model(see \cite{KOS06}, \cite{KO07}), a generic conductance function is expected to give rise to a limit shape where there are macroscopic regions corresponding to each lattice point in the Newton polygon of the resistor network(see sections \ref{vbl} and \ref{egm}). Figure \ref{et} suggests that although the class of conductances functions we consider lies in a closed subvariety of $\Z^2$-periodic conductances, it is still sufficiently general to exhibit all the possible macroscopic phases.

\textbf{Acknowledgements}. I would like to thank Richard Kenyon for suggesting this problem and many helpful discussions that led to this paper.

\section{Groves and the cube recurrence}\label{1}
\subsection{Groves}

\smallskip

\smallskip

We recall some some terminology and basic properties of groves from \cite{CS04}. A \textit{rhombus} is any set in one of the following three forms for $(i,j,k) \in \mathbb{Z}_{\leq 0} ^3$:
$$
r_a(i,j,k):=\{(i,j,k),(i,j-1,k),(i,j,k-1),(i,j-1,k-1)\}
$$
$$
r_b(i,j,k):=\{(i-1,j,k),(i,j,k),(i,j,k-1),(i-1,j-1,k)\}
$$
$$
r_c(i,j,k):=\{(i,j,k),(i-1,j,k),(i,j-1,k),(i-1,j-1,k)\}.
$$
We call the edges $E_q(i,j,k):=\{(i,j-1,k),(i,j,k-1)\}$ and $e_q(i,j,k):=\{(i,j,k),(i,j-1,k-1)\}$ the \textit{long diagonal} and the \textit{short diagonal} of the rhombus $r_q(i,j,k)$ respectively for $q \in \{a,b,c\}$. We denote the set of all diagonals of rhombi by $\mathcal{D}$.

\smallskip
\begin{figure}
\centering
\includegraphics[width=0.3\textwidth]{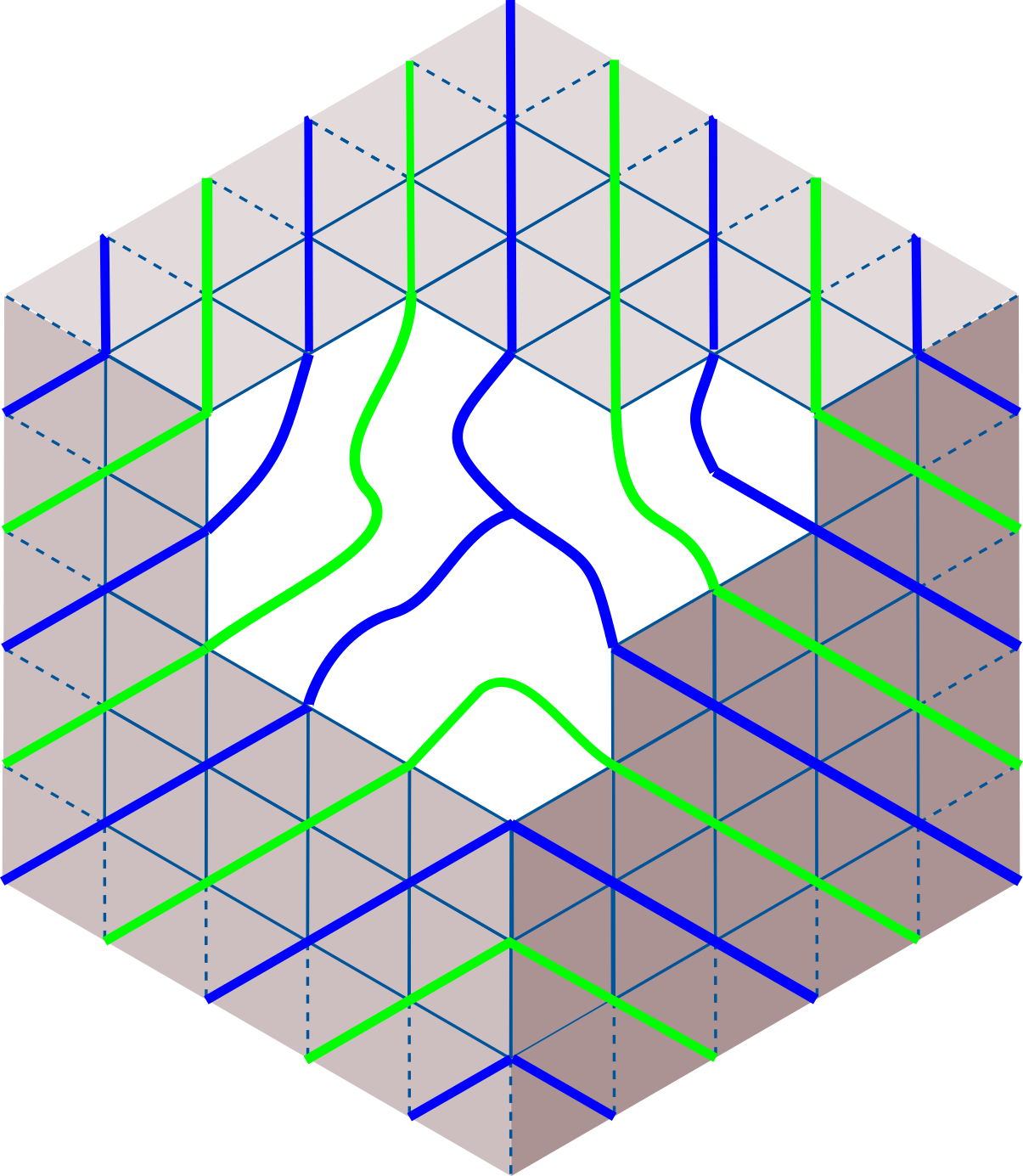}
\caption{The connectivity of a grove.}\label{conn}
\end{figure}

Let $\Gamma_{\mathcal{I}}$ be the graph with vertex set $\mathcal{I}$ and edges the long and short diagonals appearing in each rhombus in $\mathcal{I}$. Then an \textit{$\mathcal{I}$-grove} is a subgraph $G$ of $\Gamma_{\mathcal{I}}$ with the following properties:

\begin{itemize}
\item The vertex set of $G$ is all of $\mathcal{I}$.
\item For each rhombus in $\mathcal{I}$, exactly one of the two diagonals occurs in $G$.
\item There exists $N$ such that if all the vertices of a rhombus satisfies $i+j+k <- N$, the short diagonal occurs.
\item For $N$ large enough, every component of $G$ contains exactly one of the following sets of vertices, and each such set is contained in a component of $G$ (Figure \ref{conn}),
\begin{itemize}
\item $\{(0, p, q),(p, 0, q)\}$, $\{(p, q, 0),(0, q, p)\}$, and $\{(q, 0, p),(q, p, 0)\}$ for all $p, q$ with
$0 > p > q$ and $p + q \in \{-N-1,-N-2 \}$;
\item $\{(0, p, p),(p, 0, p),(p, p, 0)\}$ for $2p \in \{-N - 1, -N - 2\}$;
\item $\{(0, 0, q)\}$, $\{(0, q, 0)\}$, and $\{(q, 0, 0)\}$ for $q \leq -N - 1$.

\end{itemize}

\end{itemize}

\smallskip
It is shown in \cite{CS04} that groves on standard initial conditions $I(n)$ are completely determined by their long diagonal edges. Therefore, we can represent groves as a spanning forest of a finite portion of the triangular lattice(see Figure \ref{grove}), which is called a \textit{simplified grove}.
 
\smallskip

Suppose $\mathcal{I} \in \mathfrak{I}$ is a set of initial conditions. The edge-variables version of the cube recurrence gives $g_{0,0,0}$ as a rational function in the variables 
$\{a_{j,k},b_{i,k},c_{i,j},g_{i,j,k}\}_{(i,j,k)\in \mathcal{I}}$. The following is the main result of \cite{CS04}.
\begin{theorem}[\cite{CS04}]\label{theorem1}
$$g_{0,0,0}=\sum_{G \in \mathcal{G}(\mathcal{I})}M(G),$$
where 
$$M(G)=\left(\prod_{e_a(i,j,k)\in E(G)}a_{j,k}\right)\left(\prod_{e_b(i,j,k)\in E(G)}b_{i,k}\right)\left(\prod_{e_c(i,j,k)\in E(G)}c_{i,j}\right)m_g(G)$$
and 
$$
m_g(G)=\prod_{(i,j,k)\in \mathcal{I} }g_{i,j,k}^{deg(i,j,k)-2},
$$
where $deg(i,j,k)$ is the degree of the vertex $(i,j,k)$ in the grove $G$.
\end{theorem}

Now suppose $f:\mathbb{Z}_{\leq 0} \rightarrow \mathbb{R}_{>0}$ satisfies the cube recurrence. Since $f_{i,j,k}$ are positive real numbers, by Theorem \ref{theorem1},
$$
\mathbb{P}^f_{\mathcal{I}}(G)=\frac{m_f(G)}{f_{0,0,0}}
$$
defines a probability measure on $\mathcal{G}(\mathcal{I})$. Therefore any function $f$ that satisfies the cube recurrence induces a family of probability measures $\{\mathbb{P}^f_{\mathcal{I}}\}_{\mathcal{I} \in \mathfrak{I}}$.

\subsection{Conductance variables and the $Y-\Delta$ transformation} \label{2}

A function $C:\mathcal{D} \ra \mathbb{C}$ satisfying $C(E_q(i,j,k))=1/C(e_q(i,j,k))$ for all $q \in \{a,b,c\}, (i,j,k)\in \Z_{\leq 0}$ is called a \textit{conductance function}. We simplify notation by writing $C_q(i,j,k)=C(E_q(i,j,k))$ and $c_q(i,j,k)=C(e_q(i,j,k))$. A positive real valued conductance function $C$ determines a family of Boltzmann probability measures on groves
$\{\mathbb{P}^C_{\mathcal{I}}\}_{\mathcal{I} \in \mathfrak{I}}$:
$$
\mathbb{P}^C_{\mathcal{I}}(G):=\frac{w(G)}{Z},
$$
for $G \in \mathcal{G}(\mathcal I)$, where $w(G)=\prod_{E_q(i,j,k) \in G}C_q(i,j,k)$ is the product of the conductances of the long edges appearing in $G$ and $Z_{\mathcal I}$ is the partition function,
$$
Z_{\mathcal{I}}=\sum_{G \in \mathcal{G}(\mathcal{I})}\prod_{E_q(i,j,k) \in G}C_q(i,j,k).
$$
A conductance function $C$ is \textit{$Y-\Delta$ consistent} if for all $q,(i,j,k)$, we have
\begin{align}
C_a(i,j,k)c_a(i-1,j,k)
&=C_a(i,j,k)C_b(i,j,k)+C_a(i,j,k)C_c(i,j,k)+C_b(i,j,k)C_c(i,j,k),\nonumber\\
C_b(i,j,k)c_b(i,j-1,k)
&=C_a(i,j,k)C_b(i,j,k)+C_a(i,j,k)C_c(i,j,k)+C_b(i,j,k)C_c(i,j,k),\nonumber \\
C_c(i,j,k)c_c(i,j,k-1)
&=C_a(i,j,k)C_b(i,j,k)+C_a(i,j,k)C_c(i,j,k)+C_b(i,j,k)C_c(i,j,k).\label{ydeqns}
\end{align}
We will denote by $\mathcal{C}$ the set of $Y-\Delta$ consistent conductance functions.\\
\begin{figure}
\centering
\includegraphics[width=0.8\textwidth]{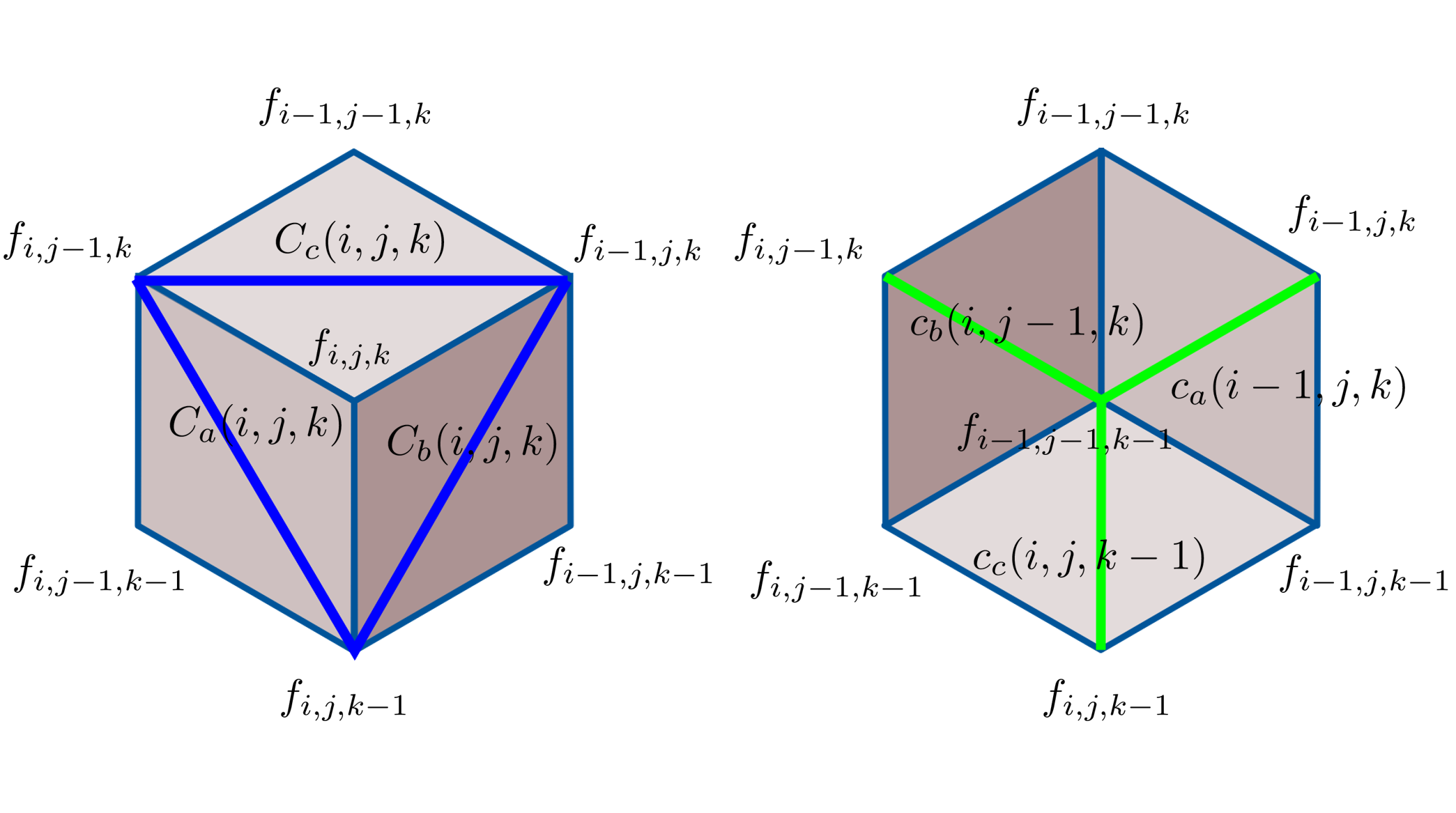}
\caption{The $Y-\Delta$ transformation.}\label{ydpic}
\end{figure}
Let $f \in \mathcal{F}$ and let us define a conductance function $C^f$(see Figure \ref{ydpic}),

\begin{align}C^f_a(i,j,k)=\frac{f_{i,j-1,k}f_{i,j,k-1}}{f_{i,j,k}f_{i,j-1,k-1}},\nonumber \\
C^f_b(i,j,k)=\frac{f_{i-1,j,k}f_{i,j,k-1}}{f_{i,j,k}f_{i-1,j,k-1}},\nonumber \\
C^f_c(i,j,k)=\frac{f_{i-1,j,k}f_{i,j-1,k}}{f_{i,j,k}f_{i-1,j-1,k}}.\label{cfeqns}\end{align}
It was observed by \cite{FZ01},\cite{GK12} that the $Y-\Delta$ equations (\ref{ydeqns}) for $C^f$ reduce to the cube recurrence for $f$, and therefore $C^f$ is $Y-\Delta$ consistent. Therefore we obtain a function 
\begin{align*}p:\mathcal{F} &\ra \mathcal{C}\\
f &\mapsto C^f.
\end{align*}
\begin{lemma}\label{ctof}$p:\mathcal{F} \ra \mathcal{C}$ is surjective.
\end{lemma}

\begin{proof}
Given $C \in \mathcal{C}$, we can construct a function $f$ such that $p(f)=C$ as follows:\\
We define for all $i,j,k \in \Z_{\leq0}$,
$$
f(i,0,0)=f(0,j,0)=f(0,0,k)=1.
$$
The equations (\ref{cfeqns}) now uniquely define $f$ on $(i,j,0),(0,j,k),(i,j,0)$ for $i,j,k \in \Z_{\leq 0}$. We define $f$ everywhere else using the cube recurrence.
\end{proof}

\subsection{Grove Shuffling} \label{3}

\begin{figure}
\centering
\subcaptionbox{\label{s1}}[.45\linewidth]
{ \includegraphics[width=0.4\textwidth]{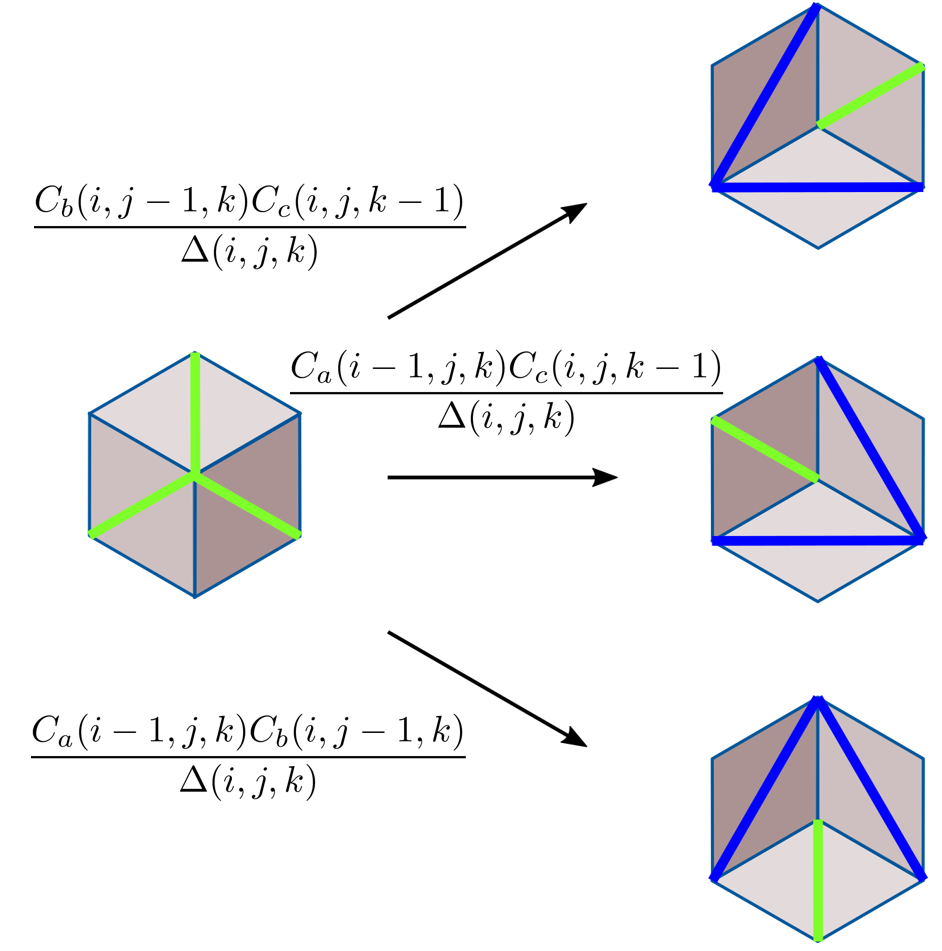}}
\subcaptionbox{\label{s2}}[1\linewidth]
{\includegraphics[width=0.4\textwidth]{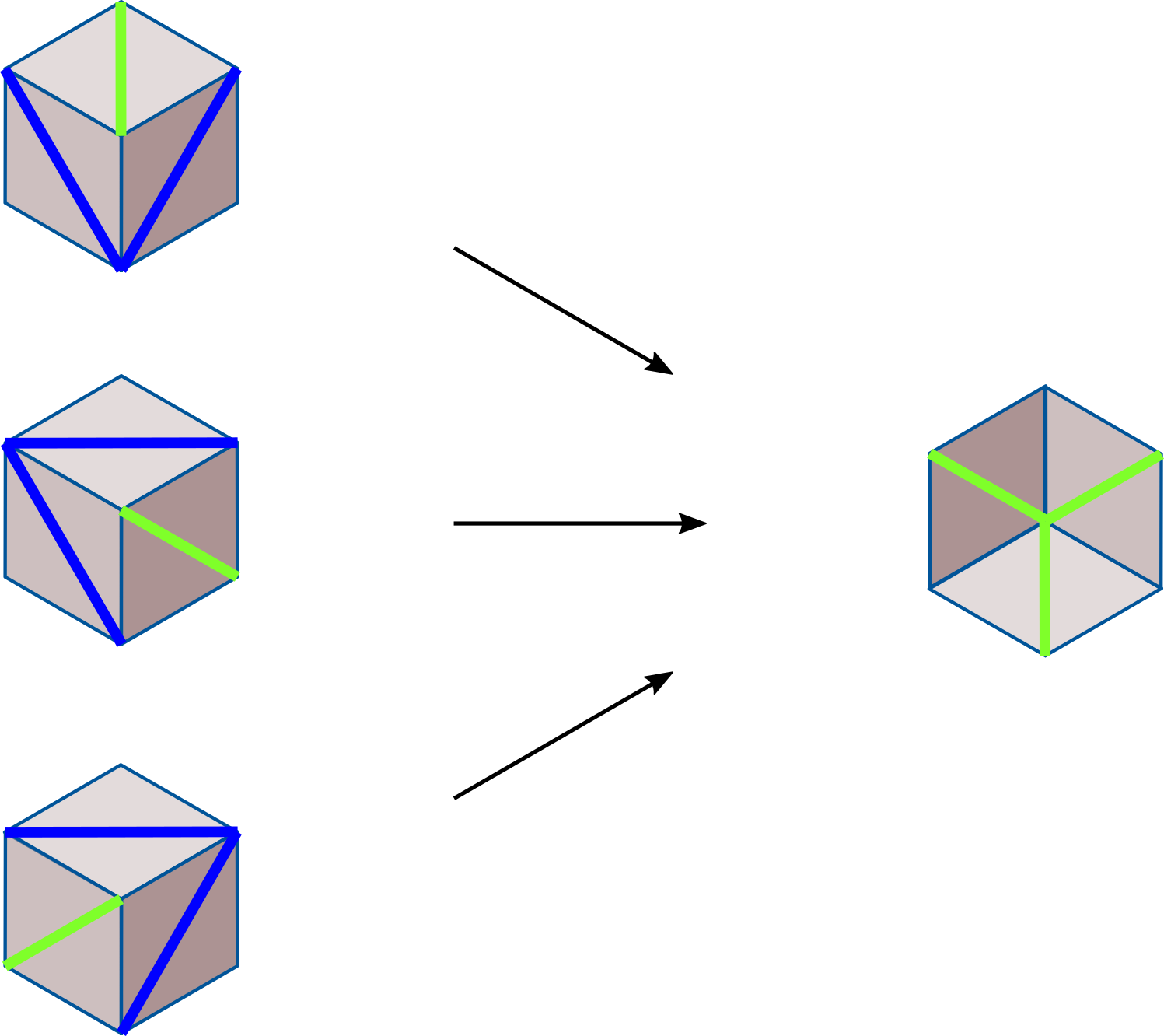}}
\subcaptionbox{\label{s3}}[1\linewidth]
{\includegraphics[width=0.4\textwidth]{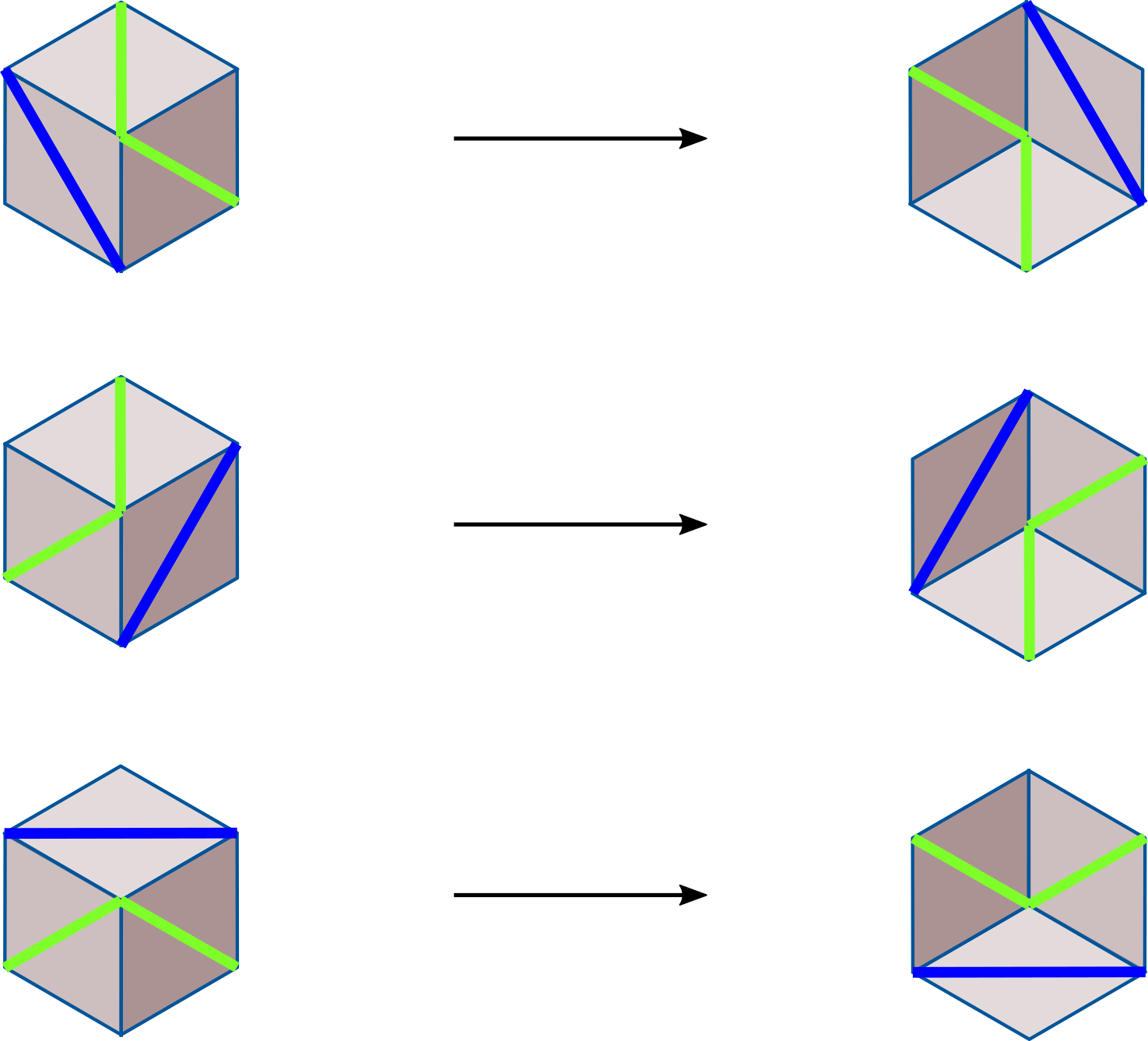}}
\caption{Grove shuffling}\label{gs}
\end{figure}

Let $C$ be a $Y-\Delta$ consistent conductance function. Let us denote 
\begin{align*}
\Delta(i,j,k)&:=(C_b(i,j-1,k)C_c(i,j,k-1)+C_a(i-1,j,k)C_c(i,j,k-1)+\\ & \quad C_a(i-1,j,k)C_b(i,j-1,k))\\&=\frac{1}{C_b(i,j,k)C_c(i,j,k)+C_a(i,j,k)C_c(i,j,k)+C_a(i,j,k)C_b(i,j,k)},
\end{align*}
where the equality is a consequence of (\ref{ydeqns}). 
We define a generalization of the edge-variables version of the cube recurrence:

\begin{align*}
g_{i,j,k}g_{i-1,j-1,k-1}&=\frac{1}{\Delta(i,j,k)}(C_b(i,j-1,k)C_c(i,j,k-1)g_{i-1,j,k}g_{i,j-1,k-1}\\&+C_a(i-1,j,k)C_c(i,j,k-1)g_{i,j-1,k}g_{i-1,j,k-1}\\&+C_a(i-1,j,k)C_b(i,j-1,k)g_{i,j,k-1}g_{i-1,j-1,k}),
\end{align*}
for $(i,j,k) \in \mathbb{Z}^3_{\leq 0}$.

Grove shuffling is a local move that generates groves with measure $\mathbb{P}^C_{\mathcal{I}}(G)$ and couples the probability measures for different initial conditions in a convenient way(see Figure \ref{gs}). Grove shuffling takes a cube, removes it and replaces a configuration on the left in Figure \ref{gs} with a corresponding configuration on the right. The only random part is (a) where the configuration on the left is replaced with one of the configurations on the right with probabilities indicated on the arrows. 

\smallskip
We can generate a random grove on initial conditions $\mathcal{I}$ as follows. Start with the unique grove on $\{(i,j,k)\in \mathbb{Z}^3_{\leq 0}\: \text{max}\{i,j,k\}=0\}$. Use grove shuffling to remove cubes till you end up with initial conditions $\mathcal{I}$. The following lemma shows that this can always be done.

\begin{lemma}[\cite{CS04}]\label{lemmain} 
Suppose $(0,0,0) \in \mathcal{I}$. Then there exist $i,j,k \leq 0$ such that $(i-1,j,k),(i,j-1,k-1),(i,j-1,k),(i-1,j,k-1),(i,j,k-1),(i-1,j-1,k) \in \mathcal{I} $ (and so $(i,j,k)\in \mathcal{U}$)
\end{lemma}

The key input in the computation of the limit shape is the following theorem that generalizes Theorem \ref{theorem1}.

\begin{theorem}\label{mainthm}
Suppose $C$ is a $Y-\Delta$ consistent conductance function and let $f$ be the solution to the cube recurrence such that $p(f)=C$ from lemma \ref{ctof}. The following are true:
\begin{itemize}
\item The generalized cube recurrence satisfies for all $\mathcal I \in \mathfrak I$,
$$
g_{0,0,0}=\sum_{G \in \mathcal{G}(\mathcal{I})}\mathbb{P}^C_{\mathcal{I}}(G)m_g(G).
$$
\item Grove shuffling generates groves with probability measure $\mathbb{P}^C_{\mathcal{I}}$, regardless of the order in which cubes are shuffled.
\item The probability measures $\mathbb{P}^C_{\mathcal{I}}$ and $\mathbb{P}^f_{\mathcal{I}}$ are the same.
\item $$Z_{\mathcal{I}}=\prod \Delta(i,j,k),$$ where the product is over all $(i,j,k)$ such that the cube at $(i,j,k)$ is removed to reach $\mathcal{I}$.
\end{itemize}
\end{theorem}

\begin{proof}
The proof is by induction on $|\mathcal{U}|$. If $\mathcal{U}=\emptyset$ then it is clear. Suppose $\mathcal{U}$ is not empty. Choose $(i,j,k)$ as in lemma \ref{lemmain}. We obtain the initial conditions $\mathcal{I}$ by shuffling the cube with vertex $(i,j,k)$ in $\mathcal{I}'$. We first show that 
$$
Z_{\mathcal{I}}=Z_{\mathcal{I}'}\Delta(i,j,k).
$$
Consider any $\mathcal{I}$ grove $G$. Since $(i-1,j-1,k-1)$ belongs to three rhombi, it has degree 3, 2 or 1. \\
Suppose $(i-1,j-1,k-1)$ has degree 1. Then $G$ belongs to a triple of $\mathcal{I}$ groves, say $\{G_1,G_2,G_3\}$ in the order shown in Figure \ref{gs} (a), all of which are obtained from a single $\mathcal{I'}$ grove $G'$ by shuffling. We have 
\begin{align*}
w(G_1)&=C_b(i,j-1,k)C_c(i,j,k-1)w(G'),\\
w(G_2)&=C_a(i-1,j,k)C_c(i,j,k-1)w(G'),\\
w(G_3)&=C_a(i-1,j,k)C_b(i,j-1,k)w(G').
\end{align*}
Therefore 
\begin{align*}&w(G_1)+w(G_2)+w(G_3)\\&=w(G').\Delta(i,j,k).
\end{align*}
\\
\smallskip
Suppose $(i-1,j-1,k-1)$ has degree 3. Then there are three $\mathcal{I}'$ groves, say $G_1$,$G_2$ and $G_3$ (in the order in Figure \ref{gs}(b)) that upon shuffling the cube at $(i,j,k)$ yields $G$. We have 
\begin{align*}
w(G)&=\frac{w(G_1)}{C_b(i,j,k)C_c(i,j,k)},\\
&=\frac{w(G_2)}{C_a(i,j,k)C_c(i,j,k)},\\
&=\frac{w(G_3)}{C_a(i,j,k)C_b(i,j,k)}.
\end{align*}
Therefore 
\begin{align*}
w(G)&=\frac{w(G_1)+w(G_2)+w(G_3)}{C_b(i,j,k)C_c(i,j,k)+C_a(i,j,k)C_c(i,j,k)+C_a(i,j,k)C_b(i,j,k)}\\&=(w(G_1)+w(G_2)+w(G_3))\Delta(i,j,k).
\end{align*}
Suppose $(i-1,j-1,k-1)$ has degree 2. Without loss of generality, we assume that we are in the first situation in figure $\ref{gs}(C)$. We have 
\begin{align*}
w(G)&=w(G')\frac{C_a(i,j,k-1)}{C_a(i,j,k)}\\
&=\frac{w(G')}{C_b(i,j,k)C_c(i,j,k)+C_a(i,j,k)C_c(i,j,k)+C_a(i,j,k)C_b(i,j,k)}\\&=w(G')\Delta(i,j,k).
\end{align*}
Since each of them is multiplied by the same factor, we have shown that
$$Z_{\mathcal{I}}=Z_{\mathcal{I}'}\Delta(i,j,k).$$

Now we will check that $\mathbb{P}^C=\mathbb{P}^f$. 
Suppose $(i-1,j-1,k-1)$ has degree 1 and let $\{G_1,G_2,G_3\}$ be the triple of groves obtained from a single $\mathcal{I'}$ grove $G'$ shown in Figure \ref{gs}(a). 
\begin{align*}
\mathbb{P}^f_{\mathcal{I}}(G)&=\frac{m_f(G_1)}{f_{0,0,0}}\\&=
\frac{m_f(G')}{f_{0,0,0}}\frac{f_{i-1,j,k}f_{i,j-1,k-1}}{f_{i-1,j-1,k-1}f_{i,j,k}}\\&=\mathbb{P}^f_{\mathcal{I'}}(G')\frac{f_{i-1,j,k}f_{i,j-1,k-1}}{f_{i-1,j-1,k-1}f_{i,j,k}}\\&=
\mathbb{P}^C_{\mathcal{I}'}C(G')\frac{f_{i-1,j,k}f_{i,j-1,k-1}}{f_{i-1,j-1,k-1}f_{i,j,k}}\\&=
\frac{w(G')}{Z_{\mathcal{I'}}}\frac{C_b(i,j-1,k)C_c(i,j,k-1)}{\Delta(i,j,k)}\\&=\frac{w(G)}{Z_{\mathcal I}}\\&=\mathbb{P}^C_{\mathcal{I}}(G),
\end{align*}
where we have used
$$
\frac{C_b(i,j-1,k)C_c(i,j,k-1)}{\Delta(i,j,k)}=\frac{f_{i-1,j,k}f_{i,j-1,k-1}}{f_{i-1,j-1,k-1}f_{i,j,k}},
$$
which may be checked by direct substitution. \\
Since each shuffle is independent, the probability of obtaining $G_1$ is $$\mathbb{P}^C_{\mathcal{I'}}(G').\frac{C_b(i,j-1,k)C_c(i,j,k-1)}{\Delta(i,j,k)}\\=\mathbb{P}^C_{\mathcal{I}}(G).$$
The last thing to check is that $g_{0,0,0}$ has the stated form. Let $g'_{0,0,0}$ be the expression obtained by solving the generalized cube recurrence on $\mathcal{I}'$. By induction hypothesis, 
$$
g'_{0,0,0}=\sum_{G' \in \mathcal{G}(\mathcal{I'})}\mathbb{P}^C_{\mathcal{I'}}(G')m_g(G'),
$$
and $g_{0,0,0}$ is obtained from $g'_{0,0,0}$ be substituting the generalized cube recurrence for $g_{i,j,k}$. We know that 
\begin{align*}
m_g(G_1)&=m_g(G')\frac{g_{i-1,j,k}g_{i,j-1,k-1}}{g_{i-1,j-1,k-1}g_{i,j,k}}\\
\mathbb{P}^C_{\mathcal{I}}(G_1)&=\mathbb{P}^C_{\mathcal{I'}}(G')\frac{C_b(i,j-1,k)C_c(i,j,k-1)}{\Delta(i,j,k)},
\end{align*}
and similarly for $G_2$ and $G_3$. Therefore we see that $\mathbb{P}^C_{\mathcal{I}}(G_1)m_g(G_1)+\mathbb{P}^C_{\mathcal{I}}(G_2)m_g(G_2)+\mathbb{P}^C_{\mathcal{I}}(G_3)m_g(G_3)$ is obtained from $\mathbb{P}^C_{\mathcal{I'}}(G')m_g(G')$ by substituting the generalized cube recurrence for $g_{i,j,k}$.\\

The argument when the degree of $(i-1,j-1,k-1)$ is $2$ and $3$ is similar.
\end{proof}

Let us denote by $\mathbb{E}_{\mathcal{I}}$ the expectation with respect to the measure $\mathbb{P}_{\mathcal{I}}$. We immediately obtain:
\begin{corollary}\label{cor1}
Let $(i_0,j_0,k_0) \in \mathcal{I}$.
$$
\mathbb{E}_{\mathcal{I}}\left[\text{Exponent of the variable }  g_{i_0,j_0,k_0} \text{ in }g_{0,0,0} \right] \nonumber =\frac{\partial g_{0,0,0}}{\partial g_{i_0,j_0,k_0}}\Biggr\rvert_{g\rvert_{\mathcal{I}}=1}.
$$
\end{corollary}

\subsection{Creation rates} \label{4}
Let $(i,j,k) \in \mathbb{Z}_{\leq 0}$. Let $\mathcal{I}$ be a set of initial conditions such that $r_a(i,j,k) \subset \mathcal{I}$ and let $G$ have distribution $\mathbb{P}_{\mathcal{I}}$. Define the long edge probability 
$$p(i,j,k)=\mathbb{P}_{\mathcal{I}}(E_a(i,j,k) \in G).$$
These are well defined since if $\mathcal{I}'$ is another set of initial conditions, then we can use grove shuffling to move between $\mathcal{I}$ and $\mathcal{I}'$ leaving the rhombus  $r_a(i,j,k)$ intact. Similarly define
$$q(i,j,k)=\mathbb{P}_{\mathcal{I}}(E_b(i,j,k) \in G);$$
$$r(i,j,k)=\mathbb{P}_{\mathcal{I}}(E_c(i,j,k) \in G),$$
and the \textit{creation rates} 
$$
E(i,j,k)=1-p(i,j,k)-q(i,j,k)-r(i,j,k).
$$

It was shown in \cite{PS05} that 
\begin{align*}
E(i_0,j_0,k_0)&=\mathbb{E}_{\mathcal{I}}\left[\text{Exponent of the variable }  g_{i_0,j_0,k_0} \text{ in }g_{0,0,0} \right], \nonumber \\ \label{cr}
\end{align*}
and therefore by Corollary \ref{cor1},

\begin{equation}\label{exp}
E(i_0,j_0,k_0)=\frac{\partial g_{0,0,0}}{\partial g_{i_0,j_0,k_0}}\Biggr\rvert_{g\rvert_{\mathcal{I}}=1}.
\end{equation}

The following result  lets us obtain the generating function for $p(i,j,k)$ from that of $E(i,j,k)$. Let us introduce for convenience the notation:
\begin{align*}
U(i,j,k)&=\frac{C_b(i,j-1,k)C_c(i,j,k-1)}{\Delta(i,j,k)},\\
V(i,j,k)&=\frac{C_b(i,j-1,k)C_c(i,j,k-1)}{\Delta(i,j,k)},\\
W(i,j,k)&=\frac{C_b(i,j-1,k)C_c(i,j,k-1)}{\Delta(i,j,k)}.
\end{align*}
\begin{lemma}[(\cite{PS05} Theorem 2)]\label{gfr} The edge probabilities are given recursively by
$$p(i,j,k)=p(i+1,j,k)+(V(i+1,j,k)+W(i+1,j,k))E(i+1,j,k);$$
$$q(i,j,k)=q(i,j+1,k)+(U(i,j+1,k)+W(i,j+1,k))E(i,j+1,k);$$
$$r(i,j,k)=r(i,j,k+1)+(U(i,j,k+1)+V(i,j,k+1))E(i,j,k+1).$$
\end{lemma}

\subsection{Creation rates generating functions}

Given a $Y-\Delta$ consistent conductance function $C$, for all $\mu=(i_0,j_0,k_0) \in \Z^3_{\leq 0}$, we define a conductance function $C^\mu$ by:
$$
C^\mu_q(i,j,k)=C_q(i+i_0,j+j_0,k+k_0).
$$
Let $g^\mu$ denote the corresponding solution to the generalized cube recurrence and $E^\mu,p^\mu,q^\mu,r^\mu$ the corresponding creation rates and edge probabilities. Let $F^\mu(x,y,z)=\sum_{i,j,k \geq 0} E^{\mu}(-i,-j,-k)x^i y^j z^k$ be the generating functions for the creation rates.
\begin{lemma}\label{dlem} 
Let $(i_1,j_1,k_1)$ and $(i_2,j_2,k_2)$ be such that $(i_1,j_1,k_1)$ is in the lower cone $C(i_2,j_2,k_2)$. Then 

$$
\frac{\partial g^{\mu}_{i_2,j_2,k_2}}{\partial g^{\mu}_{i_1,j_1,k_1}}\Biggr\rvert_{g^{\mu}\rvert_{\mathcal{I}}=1}=E^{\mu+(i_2,j_2,k_2)}(i_1-i_2,j_1-j_2,k_1-k_2).
$$

\end{lemma}
\begin{proof}
Translate so that $(i_2,j_2,k_2)$ goes to $(0,0,0)$.
\end{proof}

\begin{theorem}\label{genthem}
$F^\mu(x,y,z)$ satisfy the following infinite system of linear equations over $\mathbb{C}(x,y,z)$: 

\begin{align*}
&F^\mu+xyzF^{\mu+(-1,-1,-1)}-U^\mu(0,0,0)(xF^{\mu+(-1,0,0)}+yzF^{\mu+(0,-1,-1)})\\&-V^\mu(0,0,0)(yF^{\mu+(0,-1,0)}+xzF^{\mu+(-1,0,-1)})\\&-W^\mu(0,0,0)(zF^{\mu+(0,0,-1)}+xyF^{\mu+(-1,-1,0)}))=1,
\end{align*}
for all $\mu \in \Z_{\leq 0}$.
\end{theorem}
\begin{proof}
Let $(i,j,k) \in \Z^3_{\leq 0}$ and $(i_0,j_0,k_0) \in C(i,j,k)$. Differentiating the generalized cube recurrence with respect to $g^\mu(i_0,j_0,k_0)$, setting $g^\mu|_\mathcal{I}=1$ and using lemma \ref{dlem}, we obtain
\begin{align*}
&E^{\mu+(i,j,k)}(i_0-i,j_0-j,k_0-k)+E^{\mu+(i-1,j-1,k-1)}(i_0-i+1,j_0-j+1,k_0-k+1)\\&=U^\mu(i,j,k)(E^{\mu+(i-1,j,k)}(i_0-i+1,j_0-j,k_0-k)\\&\quad +E^{\mu+(i,j-1,k-1)}(i_0-i,j_0-j+1,k_0-k+1))\\&
+V^\mu(i,j,k)(E^{\mu+(i,j-1,k)}(i_0-i,j_0-j+1,k_0-k)\\&\quad+E^{\mu+(i-1,j,k-1)}(i_0-i+1,j_0-j,k_0-k+1))\\&+W^\mu(i,j,k)(E^{\mu+(i,j,k-1)}(i_0-i,j_0-j,k_0-k+1)\\&\quad+E^{\mu+(i-1,j-1,k)}(i_0-i+1,j_0-j+1,k_0-k))).
\end{align*}
Letting $i_0-i=r,j_0-j=s, k_0-k=t$ and  relabeling $\mu+(i,j,k)$ as $\mu$, we have for all $r,s,t<0, \mu \in \Z_{\leq 0}$:
\begin{align}
&E^{\mu}(r,s,t)+E^{\mu+(-1,-1,-1)}(r+1,s+1,t+1) \nonumber\\&=U^\mu(0,0,0)(E^{\mu+(-1,0,0)}(r+1,s,t)+E^{\mu+(0,-1,-1)}(r,s+1,t+1))\nonumber\\&
+V^\mu(0,0,0)(E^{\mu+(0,-1,0)}(r,s+1,t)+E^{\mu+(-1,0,-1)}(r+1,s,t+1))\nonumber\\&+W^\mu(0,0,0)(E^{\mu+(0,0,-1)}(r,s,t+1)+E^{\mu+(-1,-1,0)}(r+1,s+1,t)).\label{eeqns}
\end{align}

Upon multiplying by $x^ry^sz^t$, summing up over all $(r,s,t) \in \mathbb{Z}^3_{\leq 0}$ and checking that the boundary terms cancel, we obtain the linear equations for $F^\mu$.
\end{proof}

\section{The resistor network model on a torus}
\subsection{Quotients of the triangular lattice}

\begin{figure}
\centering
\subcaptionbox{The fundamental domain with a conductance function.}
{\includegraphics[width=0.3\textwidth]{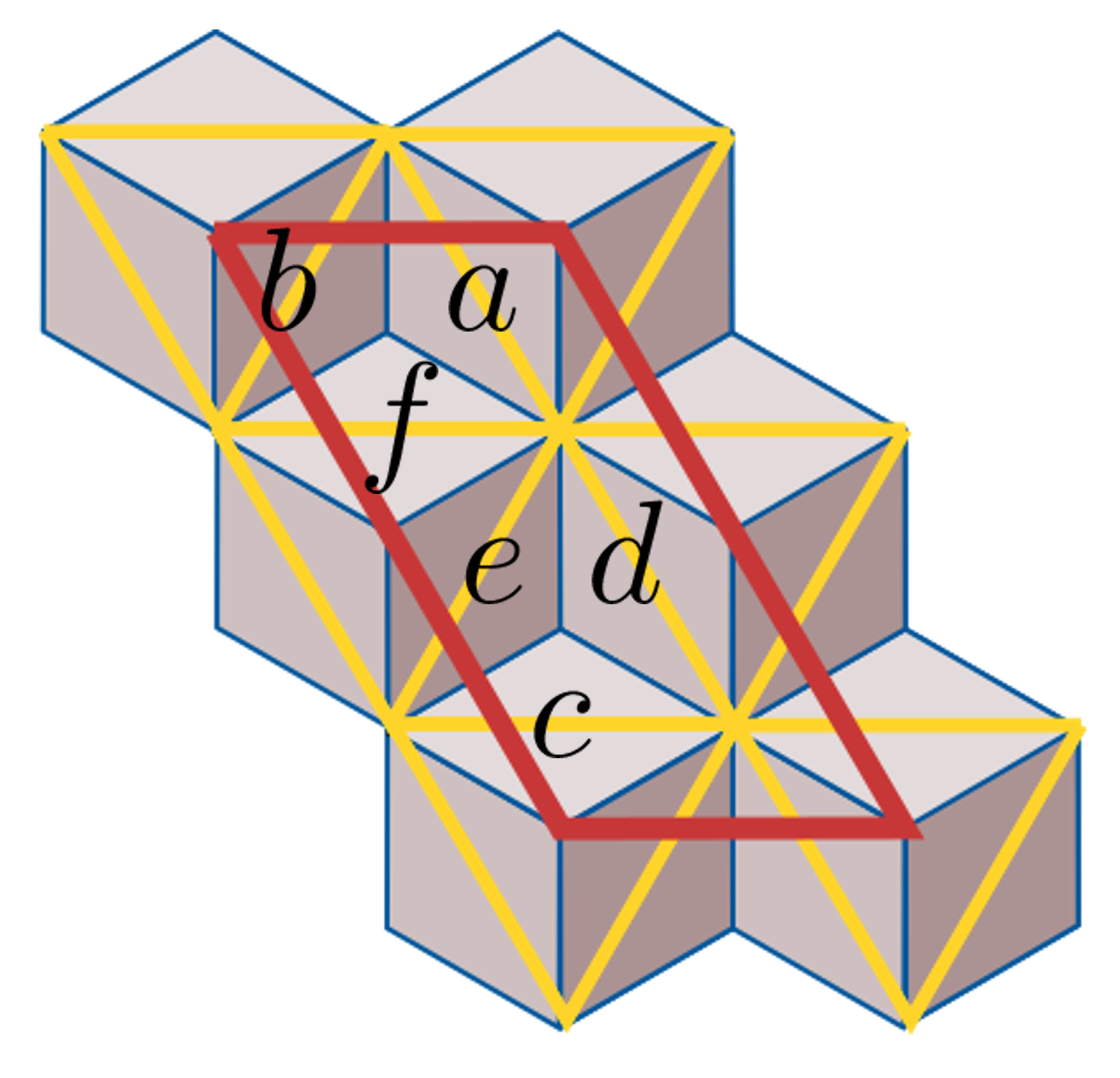}}
\subcaptionbox{Newton polygon.}
{\includegraphics[width=0.3\textwidth]{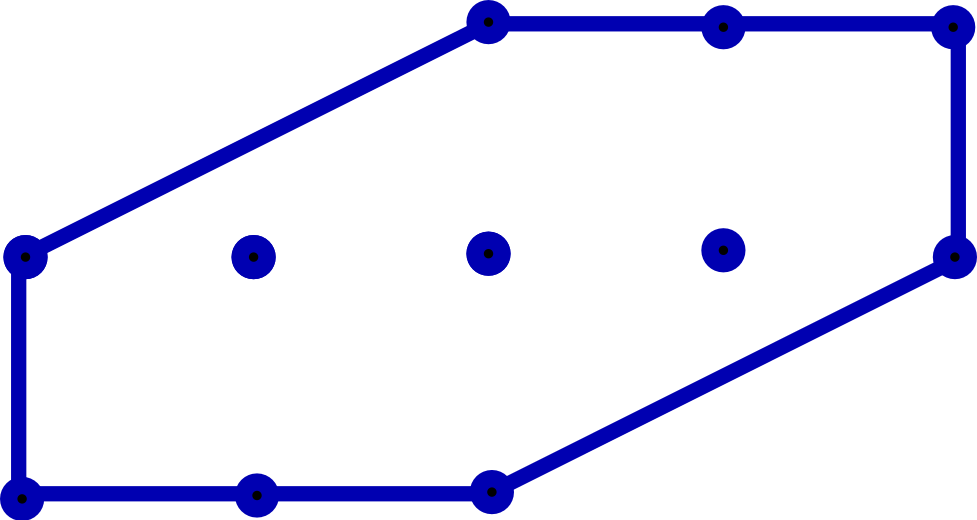}}
\caption{The graph $T_{1,2}$.}\label{fdt12}
\end{figure}

Consider the triangular lattice $T$ embedded in the plane $x+y+z=-1$ in $\mathbb{R}^3$ with vertices at $\{(i,j,k)\in \Z^3:i+j+k=0\}$. We have a $\Z^2$-action defined by 
\begin{align*}
(1,0).(i,j,k)&=(i-1,j+1,k),\\
(0,1).(i,j,k)&=(i,j+1,k-1).
\end{align*}
Let $T_{m,n}:=T/m\Z \times n\Z$ denote the quotient. It is a finite graph on a torus $\T$ with $mn$ vertices and forms an $m \times n$-cover of $T_{1,1}$. The parallelogram with vertices at $(0,0,0),(-m,m,0),(0,n,-n),(-m,m+n,-n)$ gives a fundamental domain for the torus. Figure \ref{fdt12} (a) shows the fundamental domain for $T_{1,2}$.

\subsection{The vector bundle Laplacian(\cite{K10})}\label{vbl}
Let $\Gamma$ be a finite graph on a torus embedded such that every face is a topological disk. Let $c$ be a conductance function on $\Gamma$ i.e. a real valued function on the edges of $\Gamma$ defined modulo global scaling. A pair $(\Gamma,c)$ is called a resistor network. A \textit{line bundle with connection} $(V,\phi)$ on  $\Gamma$ is the data of a complex line $V_v$ at each vertex of $\Gamma$ along with an isomorphism called \textit{parallel transport} $i_{v v'}:V_v \ra V_{v'}$ for each edge $\langle v,v' \rangle$ such that $i_{v' v}=i_{v v'}^{-1}$. Two line bundles with connection $(V,i)$ and $(V',i')$ are isomorphic if there exists isomorphisms $\psi_v:V_v \ra V'_v$ such that for all edges, the following diagram commutes.
\[
\begin{tikzcd}
V_v \arrow[r, "i_{v,v'}"] \arrow[d,"\psi_v"]
& V_{v'} \arrow[d, "\psi_{v'}" ] \\
V'_{v} \arrow[r, "i'_{v,v'}" ]
& V'_{v'}
\end{tikzcd}
\]
A connection is \textit{flat} if the monodromy around the faces of $\Gamma$ are trivial. The Laplacian is a linear operator $\Delta:\bigoplus_v V_v \ra \bigoplus_v V_v$ defined by
$$
\Delta(f)(v):= \sum_{v' \sim v}c(v,v')(f(v)-i_{v'v}f(v')).
$$

Suppose we have a flat connection with monodromy $z,w$ in the two homology directions of the torus. Then $P(z,w):=\text{det} \Delta(z,w)$ is a Laurent polynomial and is called the characteristic polynomial. The compactification of the curve $\{(z,w)\in (\C^*)^2:P(z,w)=0\}$ is called the spectral curve. The convex integral polygon 
$$
N=\text{Conv}\{(i,j)\in \Z^2: z^i w^j \text{ has non-zero coefficient in }P(z,w)\}
$$
is called the \textit{Newton polygon}. It is always centrally symmetric.\\
A zig-zag path on $\Gamma$ is an unoriented path that alternately turns maximally left or right at each vertex. Each zig-zag path gives rise to a pair of homology classes $\pm[\alpha] \in H_1(\T,\Z)$, where $[\alpha]$ is the homology of the path $\alpha$ equipped with an orientation. There is a unique centrally symmetric integral polygon $N(G)\subset H_1(\T,\Z) \cong \Z^2$ centered at the origin such that the sides of $N(G)$ are given by the vectors $\pm[\alpha]$.
\begin{lemma}\cite{GK12}
$N(G)$ coincides with the Newton polygon.
\end{lemma}

For the graphs $T_{m,n}$, we choose the connection as follows:
If an edge crosses the side $(0,0,0),(-m,m,0)$ of the fundamental parallelogram, we multiply by a factor of $w$. If an edge crosses the side $(-m,m,0),(-m,m+n,-n)$, we multiply by $z$. The Laplacian may then be represented by a matrix with entries in $\C[z^{\pm 1},w^{\pm 1}]$. The Newton polygon of $T_{m,n}$ is a hexagon with vertices at 
$$
(\pm n,0), (0,\pm m),(n,m),(-n,-m).
$$

For $T_{1,2}$ with conductance function as shown in Figure \ref{fdt12} (a), the Laplacian is 

\begin{equation*}
\Delta(z,w)=\begin{pmatrix}
a+b+d+e+f(2-z-\frac{1}{z}) & -aw-bzw-d-\frac{e}{z}\\
-\frac{a}{w}-\frac{b}{zw}-d-ez & a+b+d+e+c(2-z-\frac{1}{z}) 
\end{pmatrix},
\end{equation*}

and the Newton polygon is the hexagon in Figure \ref{fdt12} (b).

\subsection{Templerley's bijection}

\begin{figure}
\centering
\includegraphics[width=0.3\textwidth]{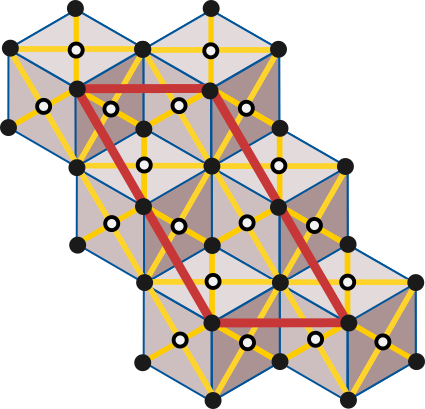}
\caption{Generalized Temperley's bijection for $T_{1,2}$.}\label{temperleys}
\end{figure}

Given a resistor network $\Gamma$ embedded on a torus $\T$, the generalized Temperley's trick \cite{KPW00} gives a bipartite graph $G_\Gamma$ on $\T$ as follows:\\
Superimpose $\Gamma$ and its dual graph, declare the vertices of $\Gamma$ and its dual black and put a white vertex at intersections of the edges of $\Gamma$ and its dual. For $T_{1,2}$, the resulting bipartite graph is shown in Figure {\ref{temperleys}}.

\subsection{Cluster Poisson variety associated to the resistor network model(\cite{GK12})}
The moduli space of line bundles with connection on $G_\Gamma$ modulo isomorphisms is denoted $\mathcal{L}_{G_\Gamma}$. Let $\hat{G_\Gamma}$ be the conjugate surface graph obtained by reversing the cyclic order of edges at each white vertex. The Poisson structure on $\mathcal{O}(\mathcal{L}_{G_\Gamma})$ is defined to be the canonical  Poisson structure on $\mathcal{O}(\mathcal{L}_{\hat{G_\Gamma}})$ coming from the intersection pairing on $\hat{G_\Gamma}$ under the natural isomorphism 

$$
\mathcal{L}_{G_\Gamma} \cong \mathcal{L}_{\hat{G_\Gamma}}.
$$

The monodromies $W_F$ around the faces of $\Gamma$ along with the monodromies around generators of $H_1(\T^2,\Z)$ form a coordinate system on $\mathcal{L}_{\Gamma}$, subject to the single relation $\prod_{F}W_F=1$.\\

A conductance function $c$ on $\Gamma$ determines $V(c)\in \mathcal{L}_{G_\Gamma}$  as follows:\\
The fiber over each vertex is identified with $\C$. The connection is defined to be the identity map if the edge comes from the incidence of a face and edge of $\Gamma$. If the edge of $G_\Gamma$ goes from a vertex of $\Gamma$ to the mid point of an edge $E$ in $\Gamma$, then the connection is defined to be $* \mapsto * \times c_E$. The moduli space of line bundles with connections arising from conductance functions forms a subvariety $\mathcal{R}_{\Gamma} \subset \mathcal{L}_{G_\Gamma}$.

A graph $\Gamma$ is\textit{ minimal} if in the universal cover, the lifts of any two zig-zag paths intersect atmost once and the lift of any zig-zag path has no self intersections.
\begin{theorem}\cite{GK12}
Any two minimal graphs with the same Newton polgygon are related by $Y-\Delta$ moves up to taking the dual graph.
\end{theorem}
 A $Y-\Delta$ transformation $\Gamma \ra \
\Gamma'$ induces a birational isomorphism 
$$
\mu_{Y-\Delta}:\mathcal{R}_{\Gamma} \dashrightarrow \mathcal{R}_{\Gamma'}.
$$

Gluing the $\mathcal{R}_{\Gamma}$ with Newton polygon $N$ using these birational maps gives the \textit{cluster Poisson variety of the resistor network model} $\mathcal{R}_N$. 

\subsection{Cluster modular transformations}\label{cmtsection}

\begin{figure}
\centering
\includegraphics[width=1\textwidth]{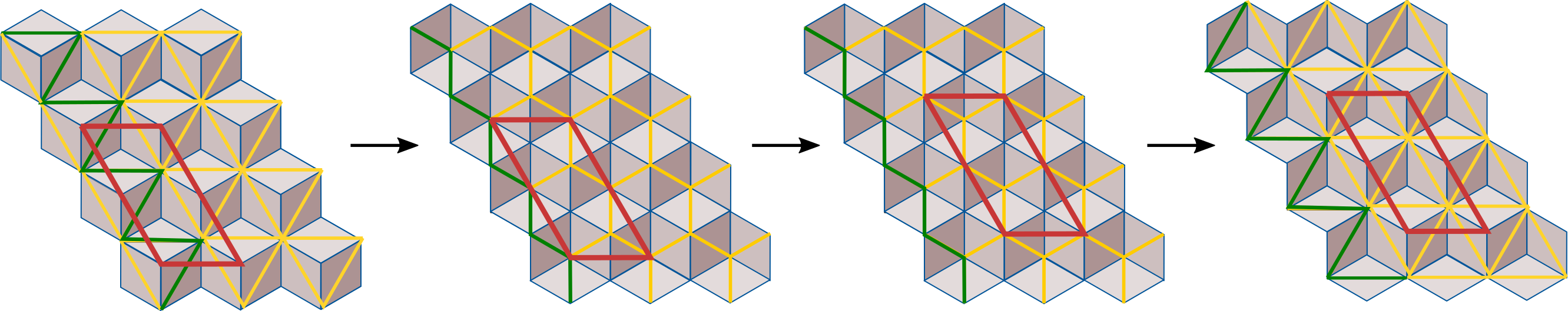}
\caption{The cluster modular transformation $T$ on $T_{1,2}$. We have colored in green a zig-zag path with homology $(0,1)$ to show that it goes around the torus. Zig-zag paths with other homology classes have no net displacement. The first step is a $Y-\Delta$ move at each downward triangle, the second step is a translation of the entire graph by the vector $(-\frac{\sqrt{3}}{2},-\frac{1}{2})$  and the third step is taking the dual graph.}\label{cmt}
\end{figure}

A birational automorphism of $\mathcal{R}_N$ induced by a sequence of $Y-\Delta$ moves taking $\Gamma$ to itself up to taking the dual graph is called a \textit{cluster modular transformation}. The group of cluster modular transformations is called the \textit{cluster modular group}(see \cite{GK12} section 6.2).\\

Cluster modular transformations are easier to describe in terms of zig-zag paths. A $Y-\Delta$ move is induced by moving a zig-zag path across the crossing of two other zig-zag paths. For each centrally symmetric pair of edges $E,E'$ of $N$, we have zig-zag paths $(\alpha_i)_{i=1}^k$ in cyclic order in the fundamental domain of the torus with homology class given by the vector of the edge $E$ up to orientation. By minimality, these paths do not intersect. Isotope them cyclically around the torus in the direction specified by the outward normal(out from $N$) to the edge vector of $E$, so that $(\alpha_1,...,\alpha_k)\ra(\alpha_2,...,\alpha_{k},\alpha_{1})$ or  $(\alpha_1,...,\alpha_k)\ra(\alpha_k,\alpha_{1},\alpha_2,...,\alpha_{k-1})$, leaving the other strands unchanged. This induces a sequence of $Y-\Delta$ moves corresponding to moving $\alpha_i$ through simple crossings of two zig-zag paths, which transforms $\Gamma$ back to itself. The composition of the birational maps induced by these $Y-\Delta$ moves gives a cluster modular transformation $T_{E}$. Note that $T_{E'}=T_E^{-1}$.\\

We are interested in the cluster modular transformation $T:=T_{\langle(-n,0),(-n,-m)\rangle}$ on the graph $T_{m,n}$. In the case of $T_{1,2}$, this cluster modular transformation is illustrated in Figure $\ref{cmt}$.
\subsection{Ergodic Gibbs measures}\label{egm}
Let $\tilde{\Gamma}$ be the lift of $\Gamma$ to the universal cover of the torus. An \textit{essential spanning forest}(ESF) on $\tilde{\Gamma}$ is a spanning forest whose every component is infinite. Kenyon proved the following classification for ergodic Gibbs measures(EGMs) on ESFs on $\tilde{\Gamma}$, extending the results in \cite{KOS06} for the dimer model to groves.
\begin{theorem}[\cite{K17}]
For each centrally symmetric pair $((s,t),(-s,-t)) \in N(P)$, there exists a unique EGM on ESFs of $\tilde{\Gamma}$ with components having average density $(s,t)$ in the two coordinate directions.  
\end{theorem}
An ergodic Gibbs measure is in the solid phase if some edge correlation is deterministic. It is in the liquid phase if the edge correlations decay quadratically and gaseous if the decay is exponential. The solid phases correspond to boundary lattice points of the Newton polygon and the gaseous phases to the interior lattice points unless the corresponding compact oval in $P(z,w)$ degenerates to a real node, in which case it is liquid. This always happens for the central point, which corresponds to the UST measure.\\
For $T_{1,2}$, the Newton polygon is in Figure {\ref{fdt12}} (b). Therefore, there are four EGMs in the solid phase and one EGM in the gaseous phase. By analogy with dimer limit shapes(see \cite{KOS06}, \cite{KO07}), we expect to see macroscopic regions where the local statistics are described by each of the solid and gaseous EGMs in a generic limit shape.

\section{Edge probability generating functions}

Start with a conductance function $C^t$ on the $T_{m,n}$ that is $N$-periodic under the cluster modular transformation $T$. We construct a $Y-\Delta$ consistent conductance function on $\Z^3$ as follows:\\
Fix a scale factor for $C^t$ and define $C|_{\{i+j+k=-1\}}=C^t$. Extend to all of $\Z^3$ using the $Y-\Delta$ transformation.\\

From $T^NC^t=C^t$ up to scaling, we have for all $k$,
\begin{align*}
U^{\mu+k(-N,0,0)}&=U^\mu,\\
V^{\mu+k(-N,0,0)}&=V^\mu,\\
W^{\mu+k(-N,0,0)}&=W^\mu,
\end{align*}
which implies that 
$$
F^{\mu+k(-N,0,0)}=F^\mu.
$$
Moreover, since $C^t$ comes from $T_{m,n}$, we also obtain for all $k \in \Z$,
\begin{align*}
C^{\mu+k(-m,m,0)}&=C^\mu,\\
C^{\mu+k(0,0,n)}&=C^\mu,
\end{align*}
from which we get,
\begin{align*}F^{\mu+k(-m,m,0)}&=F^\mu,\\
F^{\mu+k(0,0,n)}&=F^\mu.
\end{align*}
Let us introduce an equivalence relation $\sim$ on $\Z^3$: For all $k$,
\begin{align*}
\mu &\sim \mu+k(-N,0,0),\\
\mu &\sim \mu+k(-m,m,0),\\
\mu & \sim \mu+k(0,0,n).
\end{align*}
$\mathcal{M}:=\Z^3/\sim$ parameterizes the distinct $F^\mu$. The infinite linear system of equations in Theorem \ref{genthem} reduces to a finite linear system and so we obtain a matrix $A=(A_{[\mu],[\nu]})$ for $[\mu],[\nu] \in \mathcal{M}$, such that the linear system may be written as
$$
A(F^{[\mu]})_{{[\mu]} \in \mathcal{M}}=\mathbf{1},
$$
where $\mathbf{1}$ is the constant vector of $1$s.\\
Let 
\begin{align*}
G^\mu_p(x,y,z)&=\sum_{i,j,k \geq 0} p^{\mu}(-i,-j,-k)x^i y^j z^k,\\
G^\mu_q(x,y,z)&=\sum_{i,j,k \geq 0} q^{\mu}(-i,-j,-k)x^i y^j z^k,\\
G^\mu_r(x,y,z)&=\sum_{i,j,k \geq 0} r^{\mu}(-i,-j,-k)x^i y^j z^k,
\end{align*}
be the generating functions for edge probabilities. 
\begin{theorem}\label{pgenthm}
The edge probability generating functions satisfy the following linear system of equations:
\begin{align}
A \left(G^{[\mu]}_p\right)_{{[\mu]} \in \mathcal{M}}&=\frac{x}{1-x}(Q^{[\mu]}(0,0,0)+R^{[\mu]}(0,0,0))_{{[\mu]} \in \mathcal{M}},\nonumber\\
A \left(G^{[\mu]}_p\right)_{{[\mu]} \in \mathcal{M}}&=\frac{y}{1-y}(P^{[\mu]}(0,0,0)+R^{[\mu]}(0,0,0))_{{[\mu]} \in \mathcal{M}},\nonumber\\
A \left(G^{[\mu]}_p\right)_{{[\mu]} \in \mathcal{M}}&=\frac{z}{1-z}(P^{[\mu]}(0,0,0)+Q^{[\mu]}(0,0,0))_{{[\mu]} \in \mathcal{M}}.\label{pgen}
\end{align}
\end{theorem}
\begin{proof}
We will derive the first equation, the other two may be derived in the same way. Let $\alpha^{[\mu]}(i,j,k)=p^{[\mu]}(i-1,j,k)-p^{[\mu]}(i,j,k)$. By Lemma \ref{gfr}, $\alpha^{[\mu]}(i,j,k)=(V^{[\mu]}(i,j,k)+W^{[\mu]}(i,j,k))E^{[\mu]}(i,j,k)$. 
We have
\begin{align*}&\alpha^{[\mu-v]}(r+v,s+v,t+v)\\&=(V^{[\mu-v]}(r+v,s+v,t+v)+W^{[\mu-v]}(r+v,s+v,t+v))E^{[\mu-v]}(r+v,s+v,t+v)\\&=(V^{[\mu]}(r,s,t)+W^{[\mu]}(r,s,t))E^{[\mu-v]}(r+v,s+v,t+v).
\end{align*}
In particular, we observe that the factor $(V^{[\mu]}(r,s,t)+W^{[\mu]}(r,s,t))$ does not depend on $v$. Therefore from equation (\ref{eeqns}), we obtain
\begin{align*}
&\alpha^{[\mu]}(r,s,t)+\alpha^{[\mu+(-1,-1,-1)]}(r+1,s+1,t+1) \nonumber\\&=U^\mu(0,0,0)(\alpha^{[\mu+(-1,0,0)]}(r+1,s,t)+\alpha^{[\mu+(0,-1,-1)]}(r,s+1,t+1))\nonumber\\&
+V^\mu(0,0,0)(\alpha^{[\mu+(0,-1,0)]}(r,s+1,t)+\alpha^{[\mu+(-1,0,-1)]}(r+1,s,t+1))\nonumber\\&+W^\mu(0,0,0)(\alpha^{[\mu+(0,0,-1)]}(r,s,t+1)+\alpha^{[\mu+(-1,-1,0)]}(r+1,s+1,t)).
\end{align*}

Therefore the generating functions $H^{[\mu]}(x,y,z)=\sum_{i,j,k \geq 0} \alpha^{[\mu]}(-i,-j,-k)x^i y^j z^k$ satisfy the linear system of equations,
$$
A (H^{[\mu]})_{{[\mu]} \in \mathcal{M}}=(Q^{[\mu]}(0,0,0)+R^{[\mu]}(0,0,0))_{{[\mu]} \in \mathcal{M}}.
$$
From $\alpha^{[\mu]}(i,j,k)=p^{[\mu]}(i-1,j,k)-p^{[\mu]}(i,j,k)$, we have 
$$
G^{[\mu]}_p(x,y,z)=\frac{x}{1-x}H^{[\mu]}(x,y,z)+\sum_{(0,j,k)\in \Z^3_{\geq 0}}p^{[\mu]}(0,-j,-k)y^jz^k.
$$
Observe that for all $j,k \geq 0$, $p^{[\mu]}(0,-j,-k)=0$ and therefore we get $$\sum_{(0,j,k)\in \Z^3_{\geq 0}}p^{[\mu]}(0,-j,-k)y^jz^k=0.$$ 

\end{proof}

\section{Arctic curves}

Following the theory of asymptotics of multivariate generating functions developed in (\cite{PW02},\cite{PW04},\cite{BP11} and \cite{PW13}), we compute asymptotic edge probabilities.

Solving the linear system (\ref{pgen}), we obtain 
$$
G^{(0,0,0)}_p(x,y,z)=\frac{x}{1-x}\frac{P_p(x,y,z)}{Q(x,y,z)},
$$
where $P_p$ and $Q$ are polynomials and $Q=\text{det}(A)$. Note that the matrix $A$ is always singular at $x=1,y=1,z=1$ because the sum of the columns of $A$ vanishes. We denote by $\tilde{P}$ and $\tilde{Q}$ the homogeneous parts of these polynomials at the singular point $(1,1,1)$.

We are interested in the behavior of the coefficients $p^{[(0,0,0)]}(-i,-j,-k)$ of $G^{[(0,0,0)]}(x,y,z)$ for $(i,j,k)$ large i.e. we are interested in computing the limit,
$$
p(\bm{\hat{r}})=\lim_{\frac{(-i,-j,-k)}{\sqrt{i^2+j^2+k^2}} \ra \bm{\hat{r}}}p^{[(0,0,0)]}(-i,-j,-k),
$$
for $\bm{\hat{r}}\in \mathbb{R}^3_{\leq 0}$ satisfying $|\bm{\hat{r}}|=1$. 

For a homogeneous polynomial $f(x,y,z)$ in three variables, let $Z(f)$ be the plane curve $\{P \in \mathbb{P}^2_{\C}:f(P)=0\}$ and let $C(f)\subset \C^3$ be the affine cone over $Z(f)$. 
 The dual cone to $C(f)$ is denoted $C^\vee(f)$ and is equal to $C(f^\vee)$ where by $f^\vee$ we mean the projective dual of $f$, which may be computed by setting $z=-ux-vy$ in $f(x,y,z)$ and eliminating $x$ and $y$ from the simultaneous equations,
$$ 
f=0 ,\frac{\partial{f}}{\partial x}=0,
\frac{\partial{f}}{\partial y}=0.
$$.

The computation of asymptotic edge probabilities leads to explicit expressions for arctic curves.
We consider simplified groves on standard initial conditions of order $n$, so that they are supported on an equilateral triangle in the plane $i+j+k=-n$ with vertices at $(-n,0,0),(0,-n,0)$ and $(0,0,-n)$. We rescale so that the vertices are now at $(-1,0,0),(0,-1,0)$ and $(0,0,-1)$, obtaining an equilateral triangle $\nabla$ in the plane $i+j+k=-1$. For $n$ large, we observe macroscopic regions in the triangle with different qualitative behavior(see Figures \ref{et},\ref{unifgrove} and \ref{t12n1}). The arctic curve is the boundary separating the macroscopic regions in different phases.

\subsection{$T_{1,1}$}
\begin{figure}
\centering
\subcaptionbox{Simulation of a uniformly random grove on $\mathcal{I}(100)$.}
{\includegraphics[width=0.4\textwidth]{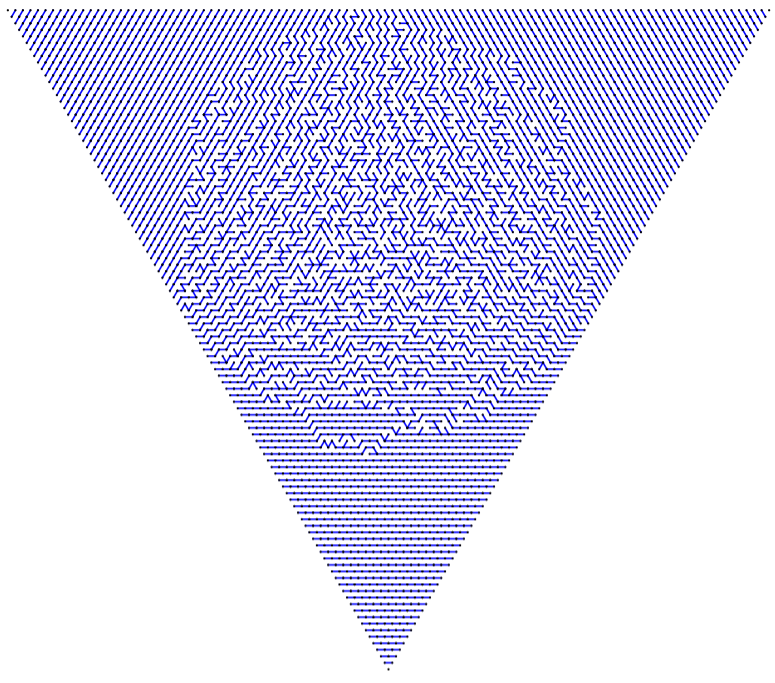}}
\subcaptionbox{The arctic circle.}
{\includegraphics[width=0.467\textwidth]{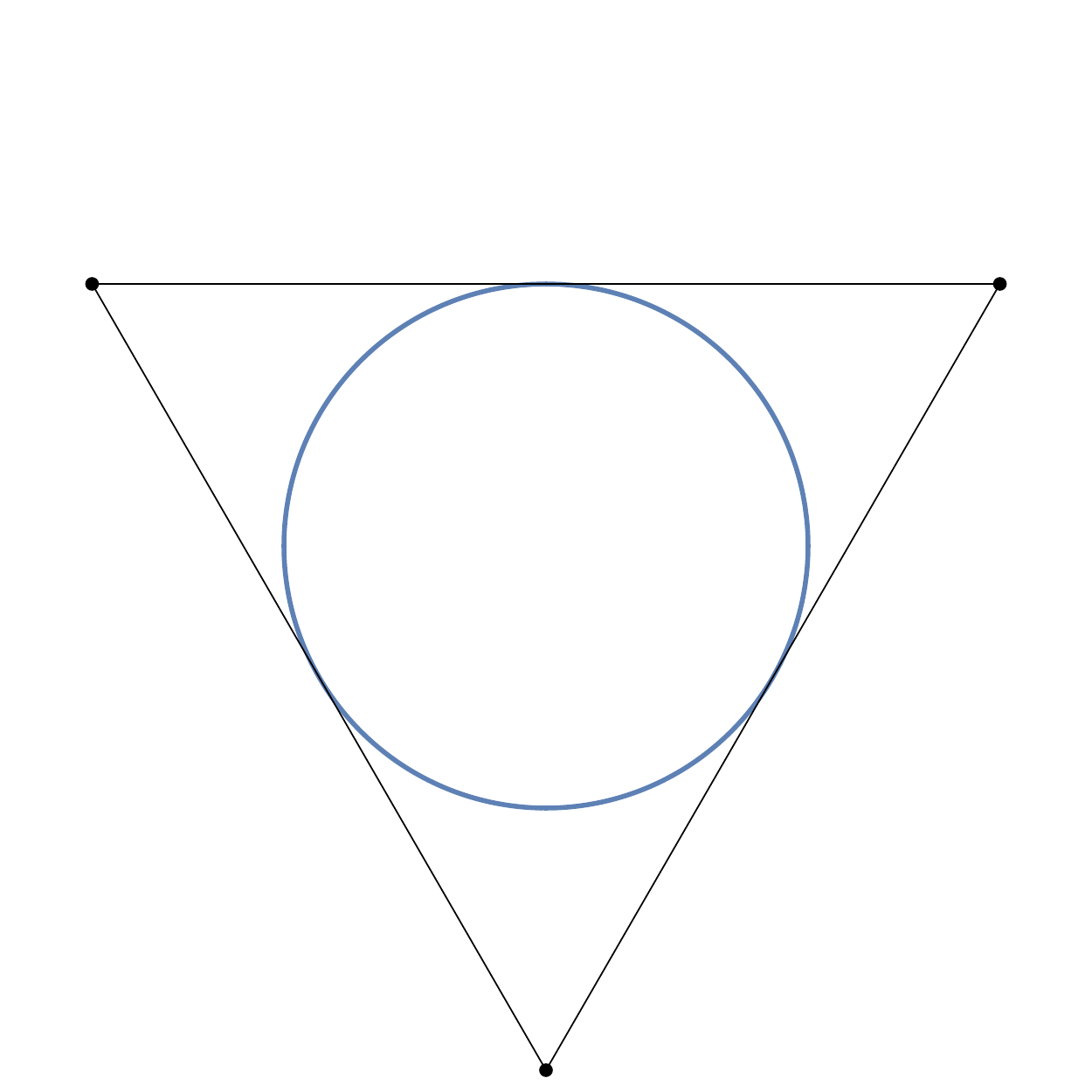}}
\caption{Uniform groves on $T_{1,1}$.}\label{unifgrove}
\end{figure}

On $T_{1,1}$, $N=1$ is forced. Let us take the conductance function on $T_{1,1}$ to be the constant function $1$. This gives rise to the uniform probability measure on groves. See Figure \ref{unifgrove} for a simulation of a random (simplified)grove on standard initial conditions of order 100. Equation (\ref{pgen}) gives
$$
G^{[(0,0,0)]}_p(x,y,z)=\frac{2x}{3(1-x)}\frac{1}{1+xyz-\frac{1}{3}(x+y+z+yz+xz+xy)}.
$$
Here $P_p(x,y,z)=\frac{2}{3}$ and $Q(x,y,z)=1+xyz-\frac{1}{3}(x+y+z+yz+xz+xy)$. The homogeneous parts at the singular point $(1,1,1)$ are
\begin{align}
\tilde{P_p}(x,y,z)&=\frac{2}{3},\nonumber \\
\tilde{Q}(x,y,z)&=\frac{2}{3} (yz+xz+xy).\label{pq1}
\end{align}
The dual curve is $\tilde{Q}^\vee(u,v,w)=vw+uw+uv-\frac{1}{2}(u^2+v^2+w^2)$. Let $K$ be the region bounded by the cone $C(\tilde{Q}^\vee)$.
\begin{theorem}[The (weak) arctic circle theorem, \cite{PS05}] \label{wact}
$p(-i,-j,-k) \ra 0$ exponentially fast outside $\text{convex-hull}(K \cup \{(u,v,w)\in \mathbb{R}^3:v=w=0\})$.
\end{theorem}

Let us denote by $P(\bm{\hat{r}})$ the point in $\nabla$ obtained by intersecting the line in the direction $\bm{\hat{r}}$ with the plane $u+v+w=-1$. $\bm{\hat{r}} \mapsto P(\bm{\hat{r}})$ is clearly a bijection. Let $C^\vee$ be the curve inscribed in $\nabla$ obtained by the intersection of $C(\tilde{Q}^\vee)$ with $u+v+w=-1$. Observe that for a point $P(\bm{\hat{r}})$ outside the region bounded by $C^\vee$, there are two (real) tangents through $P(\bm{\hat{r}})$ to $C^\vee$ while from a point inside $C^\vee$, there are no (real) tangents to $C^\vee$. What is happening is that as we approach the boundary of $C^\vee$ from the outside, the two real tangents merge into a pair of complex conjugate tangents. Under projective duality,this pair of complex conjugate tangents gives us two complex conjugate points $t_1,t_2$ on $Z(\tilde{Q})$, where we assume $t_1$ has positive imaginary part. \\
\begin{theorem}[\cite{BP11}]For $(i,j,k)\in \Z^3$ large such that for $$\bm{\hat{r}}=\frac{(-i,-j,-k)}{\sqrt{i^2+j^2+k^2}}, 
$$
$P(\bm{\hat{r}})$ is in the interior of $C^\vee$, we have 
$$
p(-i,-j,-k)=\frac{1}{2 \pi i}\int_{\delta(\bm{\hat{r}})} \omega + O\left(\frac{1}{\sqrt{i^2+j^2+k^2}}\right),
$$
where in the affine coordinates $X=\frac{x}{z},Y=\frac{y}{z}$, $\omega$ is the meromorphic 1-form\\
\begin{equation}\omega=\frac{\tilde{P_p}(X,Y,1)dX}{X \frac{\partial \tilde{Q}(X,Y,1)}{\partial Y}}.\label{omega}\end{equation} The chain of integration $\delta(\bm{\hat{r}})$ is a simple path from $t_1$ to $t_2$ passing through the arc between $[0:1:0]$ and $[0:0:1]$ containing $[1:0:0]$ on the real part of $Z(\tilde{Q})$. In particular, we have $$
p(\bm{\hat{r}})=\frac{1}{2 \pi i}\int_{\delta(\bm{\hat{r}})} \omega.
$$
\end{theorem}\label{bpthm}

Note that the only dependence on $\bm{\hat{r}}$ is through the chain of integration. Note also that this shows that $C^\vee$ is the strict boundary for exponential decay of two of the asymptotic edge probabilities. This shows that the arctic curve is $C^\vee$.

In our case, plugging in (\ref{pq1}) we obtain the 1-form
$$
\omega=\frac{dX}{X(X+1)},
$$
which has poles at $[0:1:0]$ and $[0:0:1]$ with residues $-1$ and $1$ respectively. We are lead to the following description of the arctic curve: As $P(\bm{\hat{r}})$ approaches the curve $C^\vee$, the two complex tangents from $P(\bm{\hat{r}})$ to $C^\vee$ merge into a real double tangent. Under projective duality, on $Z(\tilde{Q})$, the two points $t_1$ and $t_2$ merge into a point on the real part of $Z(\tilde{Q})$ and therefore $\delta(\bm{\hat{r}})$ becomes a closed loop. Using the residue theorem, the asymptotic edge probabilities in the frozen region $P(\bm{\hat{r}})$ approaches may be read from the residue divisor of $\omega$.  \\
If we take a non-constant conductance function on $T_{1,1}$, it was shown in \cite{PS05} that the arctic curve is an ellipse inscribed in the triangle $\nabla$.

\subsection{$T_{1,2}$ with $N=1$}
Consider the T-3-periodic conductance function on $T_{1,2}$ given in the notation of Figure \ref{fdt12} (a) by:\\
$$
a=\frac{1}{2};b=\frac{1}{8};c=\frac{3}{2};d=\frac{1}{8};e=\frac{1}{2};f=\frac{3}{2}.
$$
The linear system from (\ref{pgen}) is:
\begin{equation}\label{matrix}
\begin{pmatrix} -\frac{3 x}{16}-\frac{\text{xy}}{16}-\frac{3 y}{4}+1 & x y z-\frac{3 x z}{4}-\frac{3 y z}{16}-\frac{z}{16} \\
x y z-\frac{3 x z}{16}-\frac{3 y z}{4}-\frac{z}{16} & -\frac{x y}{16}-\frac{3 x}{4}-\frac{3 y}{16}+1 \end{pmatrix}\begin{pmatrix}G^{(0,0,0)}(x,y,z) \\ G^{(0,0,-1)}(x,y,z) \end{pmatrix}=\frac{x}{1-x}\begin{pmatrix} \frac{13}{16} \\ \frac{1}{4} \end{pmatrix}.
\end{equation}
We compute
\begin{align*}
\tilde{P}_p(x,y,z)&=\frac{185 x}{256}+\frac{13 y}{32},\\
\tilde{Q}(x,y,z)&=\frac{1}{256}(255 x^2 y + 255 x y^2 + 104 x^2 z + 370 x y z + 104 y^2 z).
\end{align*}

\begin{figure}
\centering
\subcaptionbox{$C(x\tilde{Q})\cap \{(x,y,z)\in \mathbb{R}^3:x+y+z=-1\}$ illustrating the geometry near the singular point $(1,1,1)$.}
{\includegraphics[width=0.4\textwidth]{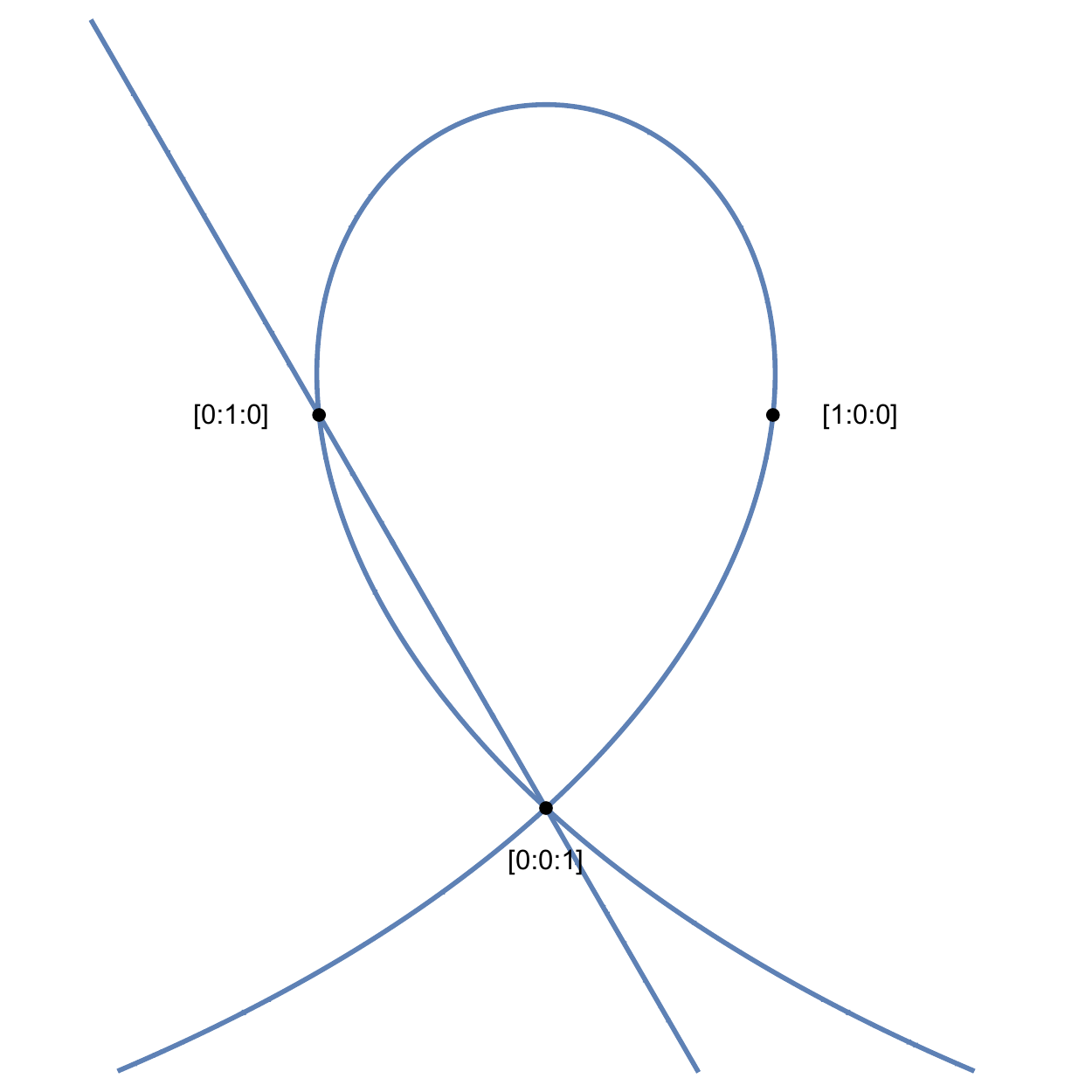}}
\subcaptionbox{The residue divisor of $\omega$ on $X \cong \mathbb{P}_\C^1$. The blue curve is the real part of $X$ and is isomorphic to $\mathbb{P}_\R^1$}
{\includegraphics[width=0.42\textwidth]{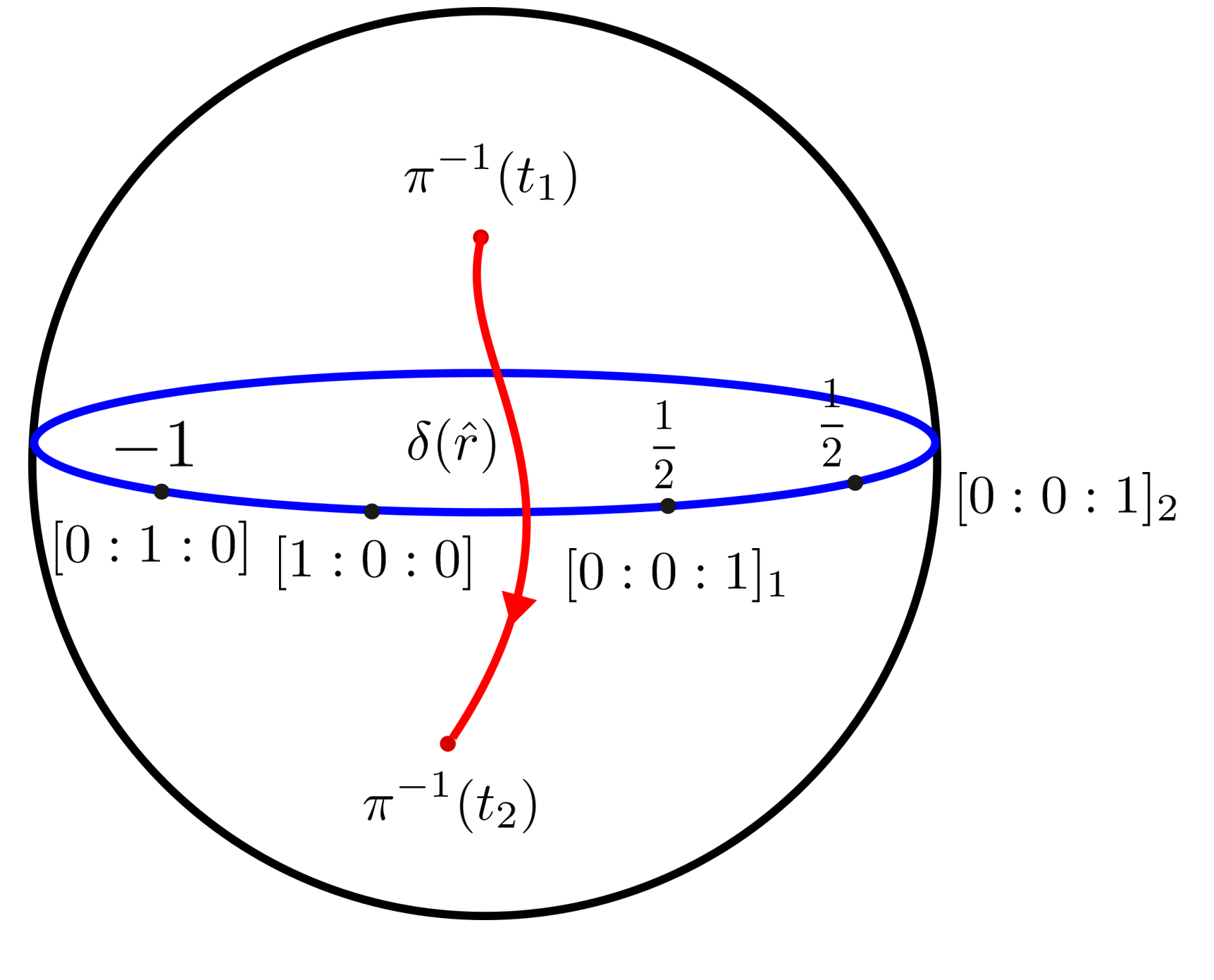}}
\caption{$\tilde{Q}(x,y,z)$ and its normalization $X$.}\label{t12qx}
\end{figure}

The dual curve is
\begin{align*}\tilde{Q}^{\vee}(u,v,w) = 6619392 u^4 - 47099520 u^3 v + 97021584 u^2 v^2 - 47099520 u v^3 + \\
 6619392 v^4 - 38301120 u^3 w - 3164400 u^2 v w - 3164400 u v^2 w- \\
 38301120 v^3 w + 73033700 u^2 w^2 + 6779600 u v w^2 + 
 73033700 v^2 w^2 \\ - 57655500 u w^3 - 57655500 v w^3 + 27635625 w^4. \nonumber
\end{align*}
\begin{figure}
\centering
\subcaptionbox{Simulation of a random grove on $\mathcal{I}(100)$.}
{\includegraphics[width=0.4\textwidth]{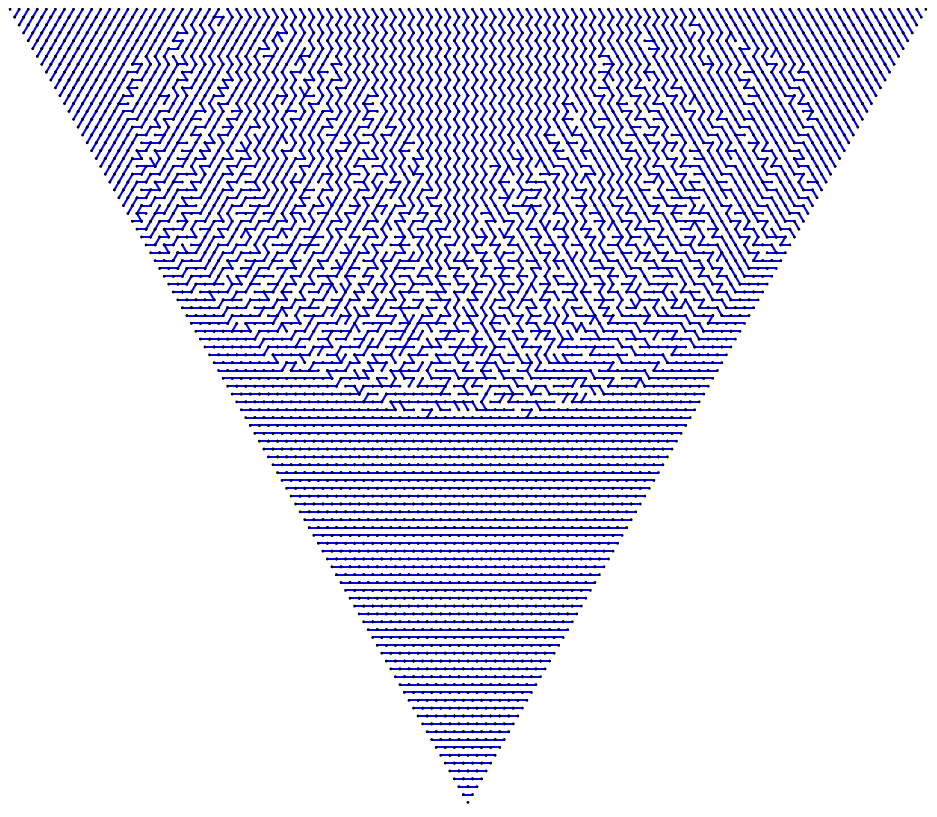}}
\subcaptionbox{The arctic curve $C^\vee$.}
{\includegraphics[width=0.42\textwidth]{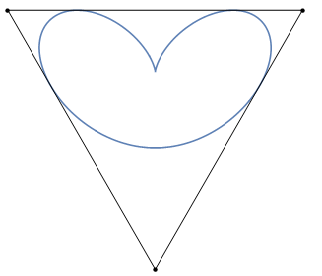}}
\caption{$T_{1,2}$ with $N=1$.}\label{t12n1}
\end{figure}

$Z(\tilde{Q})$ is singular with a node at $[0:0:1]$ (See Figure \ref{t12qx} (a)). 
This is outside the class of quadratic singularities studied in \cite{BP11}, but as observed in Section 7 of that paper, the techniques used still go through with minor modifications. Theorem \ref{wact} still holds, so we still have exponential decay outside the dual curve(see \cite{BP11} Proposition 2.23).

We need the following notions from \cite{KO07}: A degree $d$ real algebraic curve $C \subset \mathbb{P}^2_{\mathbb{R}}$ is \textit{winding} if:
\begin{itemize}
\item it intersects every line $L\subset \mathbb{P}^2_{\mathbb{R}}$ 
in at least $d-2$ points counting
multiplicity, and
\item there exists a point $p_0\in \mathbb{P}^2_{\mathbb{R}} \setminus C$  called the center, such that every line through $p_0$ intersects $C$ in $d$ points.
\end{itemize}
The dual of a winding curve $C$ is called a \textit{cloud curve}.  $C^\vee$ separates $\mathbb{P}^2_{\mathbb{R}}$ into two regions, formed by the lines that intersect $C$ in $d$ and $d-2$ points, which we call the exterior and interior respectively. A cloud curve $C^\vee$ has a unique pair of complex conjugate tangents through any point in its interior which under projective duality gives a pair of complex conjugate points on $C$. \\

Let $\pi:X \ra Z(\tilde{Q})$ be the normalization of $Z(\tilde{Q})$, where we denote by $[0:0:1]_1$ and $[0:0:1]_2$ the two points in $X$ in the fiber above the node $[0:0:1]$ of $Z(\tilde{Q})$, such that in cyclic order, we have $[0:1:0],[1:0:0],[0:0:1]_1,[0:0:1]_2$ on the real part of $X$.

\begin{theorem}\label{bpthm2}
The conclusions of Theorem \ref{bpthm} hold with the following modifications:\\
\begin{itemize}
\item The 1-form (\ref{omega}) is replaced by its pullback to $X$.
\item The chain $\delta(\bm{\hat{r}})$ is also pulled back to $X$ so that it is now a simple path from $\pi^{-1}(t_1)$ to $\pi^{-1}(t_2)$ passing through the arc between $[0:1:0]$ and $[0:0:1]_1$. 
\end{itemize}

\end{theorem}

\begin{proof}
The curve $Z(\tilde{Q})$ is a winding curve, where we may take the center to be $[1:1:1]$(This is also easily seen from the dual picture: $C^\vee$ is a cardioid (Figure \ref{t12n1} (b)) and there is a unique real tangent to $C^\vee$ from a point in its interior, whereas there are three real tangents from its exterior). Therefore, for any point $P(\bm{\hat{r}})$ in the interior of $C^\vee$ we have a pair of complex conjugate points $t_1,t_2$ on $Z(\tilde{Q})$. This is exactly the hypothesis needed in the proof of Lemma 6.15 in \cite{BP11} to determine the boundary of $\delta(\bm{\hat{r}})$.
\end{proof}

Since $$\eta=\frac{\tilde{P_p}(X,Y,1)dX}{ \frac{\partial \tilde{Q}(X,Y,1)}{\partial Y}},$$ is a holomorphic 1-form, and $\omega=\frac{1}{X}\eta$, the poles of $\omega$ are supported on the intersection of $Z(\tilde{Q})$ with the line $Z(x)$, which is a finite number of points. By computing Puiseux expansions at these points, we see that $\pi^*\omega$ has the residue divisor shown in Figure \ref{t12qx} (b). We can explain the new frozen region as follows: As $P(\bm{\hat{r}})$ approaches that region, $\pi^{-1}(t_1)$ and $\pi^{-1}(t_2)$ merge into a point on the real part of $X$ on the arc between $[0:0:1]_1$ and $[0:0:1]_2$, thereby enclosing a pole with residue $\frac{1}{2}$. 
\subsection{$T_{1,2}$ with $N=2$.}
By a simple computation, we can see that there are no $T-2-$periodic solutions that are not $T-$invariant for $T_{1,2}$, and therefore this case is subsumed by the previous one.

\subsection{$T_{1,2}$ with $N=3$}

Consider the T-3-periodic conductance function on $T_{1,2}$ given in the notation of Figure \ref{fdt12} (a) by:\\
$$
a=\frac{1}{2};b=\frac{1}{3};c=1;d=\frac{10}{3};e=\frac{1}{4};f=\frac{2}{43}.
$$

In the linear system from (\ref{pgen}), we have:
\begin{equation*}
A=\begin{pmatrix} 1& x y z& -\frac{x}{2} - \frac{y}{3}& -\frac{z}{6}& -\frac{xy}{6}& -\frac{x z}{3} - \frac{y z}{2}\\ x y z& 
  1& -\frac{5z}{6}& -\frac{20x}{129} - \frac{y}{86}& -\frac{xz}{86} - \frac{20yz}{129}& -\frac{5xy}{6}\\-\frac{43xy}{53}& -\frac{4xz}{53} - \frac{6yz}{53}& 1&  x y z& -\frac{6x}{53} - \frac{4y}{53}& -\frac{43z}{53}\\-\frac{3xz}{53} - \frac{40yz}{53}& -\frac{10xy}{53}& x y z&   1& -\frac{10z}{53}& -\frac{40x}{53}- \frac{3y}{53}\\-\frac{6x}{53} - \frac{4y}{53}& -\frac{43z}{53}& -\frac{43xy}{53}& -\frac{4xz}{53} - \frac{6yz}{53}& 1& 
  x y z\\-\frac{10z}{53}& -\frac{40x}{53} - \frac{3y}{53}& -\frac{3xz}{53} - \frac{40yz}{53}& -\frac{10xy}{53}& x y z& 1\end{pmatrix},
\end{equation*}

\begin{equation*}\left(G^{[\mu]}_p\right)_{{[\mu]} \in \mathcal{M}}=
\begin{pmatrix}G^{(0,0,0)}\\G^{(-2,0,-1)}\\G^{(-1,0,0)}\\G^{(0,0,-1)}\\G^{(-2,0,0)}\\G^{(-1,0,-1)} \end{pmatrix} \text{ and }(Q^{[\mu]}(0,0,0)+R^{[\mu]}(0,0,0))_{{[\mu]} \in \mathcal{M}}=\begin{pmatrix} \frac{1}{2}\\\frac{109}{129}\\\frac{47}{53}\\\frac{13}{53}\\\frac{47}{53}\\\frac{13}{53}\\\end{pmatrix}.
\end{equation*}

\begin{figure}
\centering
\subcaptionbox{A plot of $C(\tilde{Q})\cap \{(x,y,z)\in \mathbb{R}^3:x+y+z=-1\}$. }
{\includegraphics[width=0.6\textwidth]{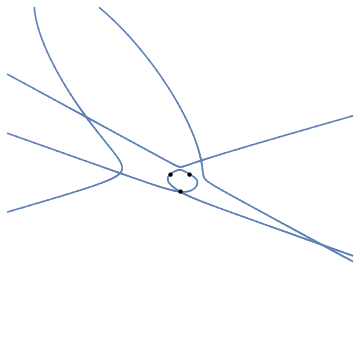}}
\subcaptionbox{The residue divisor of $\omega$ on $X$ and the chain of integration $\delta(\bm{\hat{r}})$ when $P(\bm{\hat{r}})$ is on the irreducible component bounding the gaseous region. The blue curves are the two irreducible components of $\tilde{Q}(x,y,z)$ viewed as a real algebraic curve in $\mathbb{P}^2_\R$.}
{\includegraphics[width=0.6\textwidth]{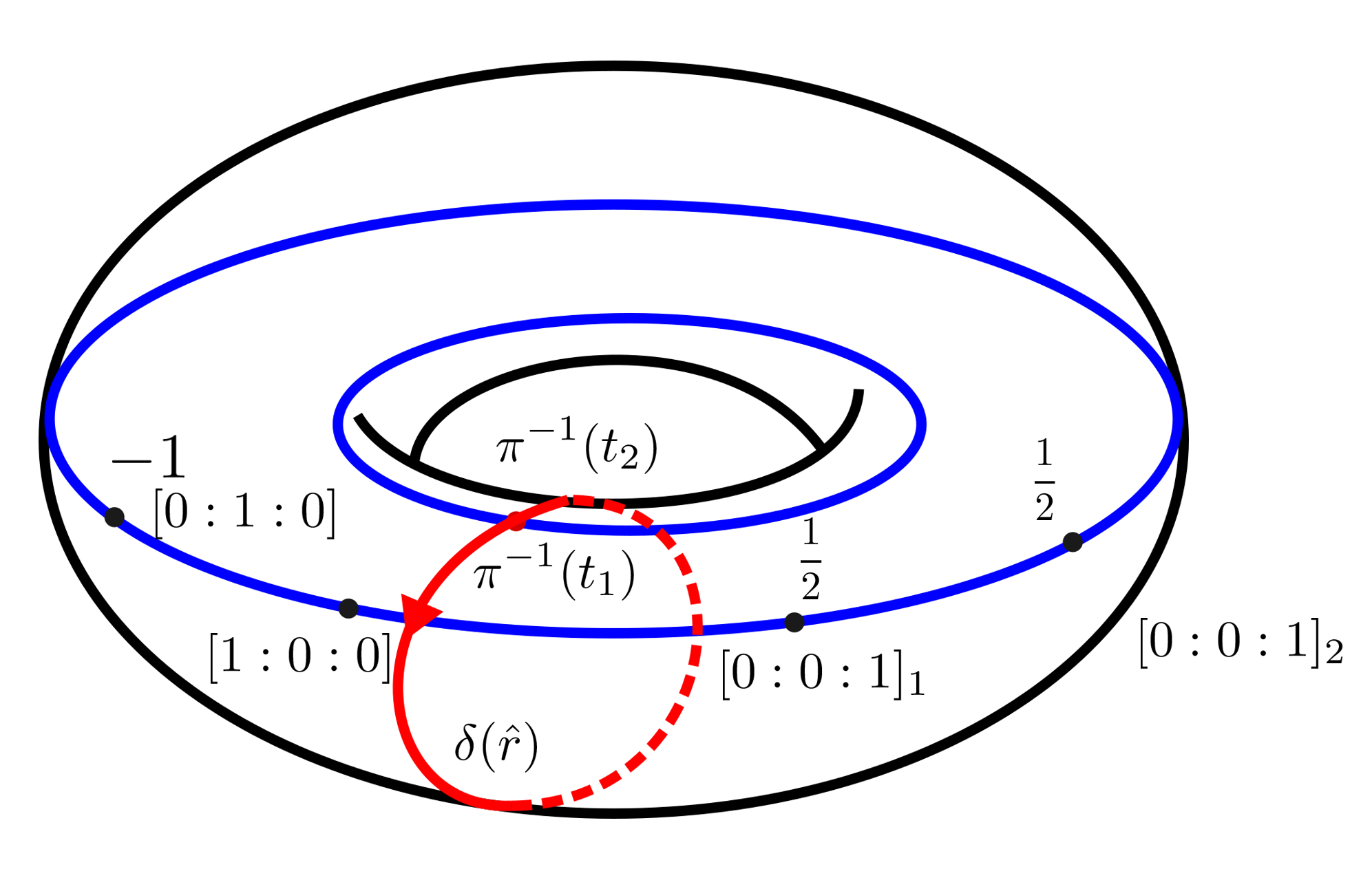}}
\caption{$\tilde{Q}(x,y,z)$ and its normalization $X$.}\label{t123qx}
\end{figure}

We obtain
\begin{align*}
\tilde{P}_p(x,y,z)&=(-8376157535 x^3-27465850948 x^2 y-37792606090 x^2 z-32422312230 x y^2\\& -81250160702 x y z-41078137290 x z^2-12081677400 y^3-37378399260 y^2 z\\& -26396541912 y
   z^2)/2035744098;
\end{align*}
\begin{align*}
\tilde{Q}(x,y,z)&=(-2195435870 x^4 y-4213162175 x^4 z-8636813573 x^3 y^2-26901515220 x^3 y z\\& -18270472400 x^3 z^2-8949558855 x^2 y^3-44782155243 x^2 y^2 z-62350371390 x^2 y
   z^2\\& -19642088100 x^2 z^3-2785734900 x y^4-25376048920 x y^3 z-53016222846 x y^2 z^2\\& -27385424860 x y z^3-4027225800 y^4 z-12459466420 y^3 z^2\\&-8798847304 y^2
   z^3)/678581366.
\end{align*}

As a real algebraic curve, we observe that $\tilde{Q}(x,y,z)$ is winding with center $(1,1,1)$ and has two irreducible components (See Figure \ref{t123qx}). Let us denote by $V_1$ the component that contains the axes and by $V_2$ the other one. Under duality, we obtain two dual real components $V_1^\vee$ and $V_2^\vee$, where $V_2^\vee$ is in the interior of $V_1^\vee$(see Figure \ref{et} (b)). The region bounded by $V_2^\vee$ is a gaseous phase. The local statistics in this region are expected to be described by the ergodic Gibbs measure of slope $(1,0)$.\\

Let $K$ be the cone over the region in the interior of $V_1^\vee$. Then it follows from \cite{BP11} (see also \cite{PW13} Theorem 11.3.8) that $p(-i,-j,-k)$ decays exponentially quickly outside $\text{convex-hull}(K \cup \{(u,v,w)\in \mathbb{R}^3:v=w=0\})$. \\

$Z(\tilde{Q})$ has genus $1$ and therefore its normalization is topologically a torus. The 1-form $\pi^* \omega$ in Theorem \ref{bpthm2} has the residue divisor shown in Figure \ref{t123qx}(b). We observe that as $P(\bm{\hat{r}})$ approaches $V_2^\vee \cap \{(u,v,w) \in \R^3: u+v+w=-1\}$, the points $\pi^{-1}(t_1)$ and $\pi^{-1}(t_2)$ merge to a point on the inverse image of $V_1$ in $X$ and therefore $\delta(\bm{\hat{r}})$ becomes a loop with non-trivial homology on the torus(see Figure \ref{t123qx}(b)).

\section{Further questions}
We are able to compute several interesting examples of arctic curves but there are several questions that remain.
\begin{itemize}
\item The projective duals of curves arising as limit shapes in the groves and in the dimer model are expected to be winding. In dimers, in cases where $\tilde{Q}(x,y,z)$ is rational, this is proved in \cite{KO07}. In \cite{K17}, groves were shown to satisfy a variational principle that is algebraically identical to the one in \cite{CKP01} for dimers, and therefore the same holds. Can we prove that the polynomials $\tilde{Q}(x,y,z)$ are winding for all genus from the generating function?

\item We have seen that Lemma 6.15 from \cite{BP11} can be extended to tackle our examples and that the necessary assumption was that $\tilde{Q}(x,y,z)$ is winding. This motivates the following problem: Extend the machinery of \cite{BP11} to describe the asymptotics of generating functions with higher degree isolated singularities where the local geometry is described by a winding curve.

\item Periods of the 1-form $\omega$ encode asymptotic probabilities of the different solid and gaseous phases. Since we know what these measures are, these asymptotic probabilities are easy to compute from and depend only on the Newton polgyon. Can we prove a description of the residue divisor of $\omega$ for general $T_{m,n}$ and $N$ in terms of the Newton polygon?

\item Can we generalize the results of this paper to groves on other $\Z^2$-periodic lattices?
\item What can we say about the subvariety of $T-N-$periodic points of $\mathcal{R}_N$?
\end{itemize}


\begin{thebibliography}{BL}
\bibitem[BP11]{BP11}Baryshnikov Y., Pemantle R. :
{\it Asymptotics of multivariate sequences, part III:
quadratic points}.
Adv. Math. 228,3127-3206.

\bibitem[CS04]{CS04}Carroll G.,
Speyer D.:
{\it The cube recurrence}.El. J. Comb 11, research paper 74.

\bibitem[CKP01]{CKP01} Cohn H., Kenyon R., Propp J. : {\it A variational principle for domino tilings . }J. Amer. Math. Soc. 14 , no. 2, 297-346.

\bibitem[FZ01]{FZ01} Fomin S., Zelevinsky A.: 
{\it Cluster algebras. I}. J. Amer. Math. Soc. 15 (2002), no. 2, 497--529. 
arXiv:math/0104151.



\bibitem[GK12]{GK12}
Goncharov A. B., Kenyon R.
\it{Dimers and cluster integrable systems}
Ann. ENS.


\bibitem[Kenn1899]{Kennelly} Kennelly A.E.: {\it Equivalence of triangles and stars in
conducting networks}, Electrical World and Engineer, 34 (1899),
413-414.

\bibitem[K10]{K10} Kenyon R.: {\it Spanning forests and the vector bundle laplacian},  
Ann. Probab. 39 (2011), no. 5, 1983?2017. arXiv:1001.4028.

\bibitem[K12]{K12} Kenyon R. : {\it The Laplacian on planar graphs and graphs on surfaces}.	arXiv:1203.1256.

\bibitem[K17]{K17} Kenyon R. :{\it Determinantal spanning forests on planar graphs},arXiv:1702.03802 .

\bibitem[KO03]{KO03} Kenyon R., Okounkov A.:  {\it 
Planar dimers and Harnack curves}. Duke Math. J. 131 (2006), no. 3, 
499--524. arXiv:math/0311062.

\bibitem[KO07]{KO07} Kenyon R., Okounkov A.:  {\it 
Limit shapes and the complex Burgers equation}. Acta Math (2007), 199: 263.


\bibitem[KOS06]{KOS06} Kenyon R., Okounkov A., Sheffield S. :
{\it Dimers and amoebae}.Ann. Math. 163 , no. 3, 1019--1056.

\bibitem[KPW00]{KPW00} Kenyon R.,Propp J., Wilson D. :
{\it Trees and matchings}.Electronic Journal of Combinatorics, 7(1):R25, 2000.


\bibitem[Miwa82]{Miwa82} Miwa T. :{\it On Hirota's difference equations.},Proc. Japan Acad. Ser. A Math. Sci. 58 (1982), no. 1, 9--12. doi:10.3792/pjaa.58.9. 

\bibitem[PS05]{PS05} Petersen T. K., Speyer D. :
{\it An arctic circle theorem for groves}.
J. Comb. Theory A, Volume 111, Issue 1ages 137-164.

\bibitem[PW02]{PW02}Pemantle R., Wilson M.:
{\it Asymptotics of multivariate sequences, part I:Smooth points of the singular variety}. J. Combin. Theory Ser. A, 97(1), 129-161.

\bibitem[PW04]{PW04} Pemantle R., Wilson M.:
{\it Asymptotics of multivariate sequences, part II:Multiple points of the singular variety}. Combin. Probab. Comput. 13735–761.

\bibitem[PW13]{PW13}Pemantle R., Wilson M.:
{\it Analytic Combinatorics in Several Variables}.
Cambridge University
Press, Cambridge, UK.

\bibitem[Propp01]{Propp} Propp J. :
{\it The Many Faces of Alternating-Sign Matrices}.Discrete Mathematics and
Theoretical Computer Science Proceedings AA (DM-CCG) (2001), 43-58.

\bibitem[Speyer07]{Speyer}Speyer D. :{\it Perfect Matchings and the Octahedron Recurrence}. Journal of Algebraic Combinatorics 25, no. 3.
\end{thebibliography}
\end{document}